\documentclass{article}
\usepackage{hyperref}

\usepackage{graphicx}%
\usepackage{multirow}%
\usepackage{amsmath,amssymb,amsfonts}%
\usepackage{amsthm}%
\usepackage{mathrsfs}%
\usepackage[title]{appendix}%
\usepackage{xcolor}%
\usepackage{textcomp}%
\usepackage{manyfoot}%
\usepackage{booktabs}%
\usepackage{algorithm}%
\usepackage{algorithmicx}%
\usepackage{algpseudocode}%
\usepackage{listings}%

\usepackage{enumitem}%
\usepackage{cleveref}%
\usepackage{bbm,bm}%

\theoremstyle{plain}%
\newtheorem{theorem}{Theorem}%
\newtheorem{proposition}{Proposition}%
\newtheorem{lemma}{Lemma}%

\theoremstyle{definition}%
\newtheorem{definition}{Definition}%
\newtheorem{example}{Example}%
\newtheorem{assumption}{Assumption}%

\theoremstyle{remark}%
\newtheorem{remark}{Remark}%

\newcommand{\bE}{\mathbb{E}}
\newcommand{\bN}{\mathbb{N}}
\newcommand{\bP}{\mathbb{P}}
\newcommand{\bR}{\mathbb{R}}
\newcommand{\bS}{\mathbb{S}}
\newcommand{\cA}{\mathcal{A}}
\newcommand{\cB}{\mathcal{B}}
\newcommand{\cI}{\mathcal{I}}
\newcommand{\cP}{\mathcal{P}}
\newcommand{\cR}{\mathcal{R}}
\newcommand{\fc}{\mathfrak{c}}
\newcommand{\rc}{\mathrm{c}}
\newcommand{\rd}{\mathrm{d}}
\newcommand{\+}[1]{{\bm{#1}}}
\newcommand{\dist}{\mathrm{dist}}
\newcommand{\eps}{\epsilon}
\newcommand{\id}{\mathrm{id}}
\newcommand{\inn}[1]{{\left\langle{#1}\right\rangle}}
\newcommand{\inns}[1]{{\langle{#1}\rangle}}
\newcommand{\norm}[1]{\left\|{#1}\right\|}
\newcommand{\od}[2]{\dfrac{\rd{#1}}{\rd{#2}}}
\newcommand{\pd}[2]{\dfrac{\partial{#1}}{\partial{#2}}}
\newcommand{\tpd}[2]{\tfrac{\partial{#1}}{\partial{#2}}}

\title{Natural Model Reduction for Kinetic Equations}

\author{Zeyu Jin \footnote{School of Mathematical Sciences, Peking University, Beijing 100871, China (\texttt{jinzy@pku.edu.cn}).} 
\and Ruo Li
\footnote{CAPT, LMAM and School of Mathematical Sciences, 
Peking University, Beijing 100871, China
(\texttt{rli@math.pku.edu.cn}).}}

\date{}

\begin{document}
  \maketitle

  \begin{abstract}
    A promising approach to investigating high-dimensional problems 
    is to identify their intrinsically low-dimensional features, 
    which can be achieved through recently developed techniques for 
    effective low-dimensional representation of functions such as 
    machine learning.
    Based on available finite-dimensional approximate 
    solution manifolds,
    this paper proposes a novel model reduction framework for kinetic 
    equations. The method employs projections 
    onto tangent bundles of approximate manifolds, naturally 
    resulting in first-order hyperbolic systems. 
    Under certain conditions on the approximate manifolds,
    the reduced models preserve several crucial properties, including 
    hyperbolicity,
    conservation laws, entropy dissipation, finite propagation 
    speed, and linear stability.
    For the first time, this paper rigorously discusses 
    the relation between the H-theorem of kinetic equations and 
    the linear stability conditions of reduced systems, 
    determining the choice of Riemannian metrics 
    involved in the model reduction.
    The framework is widely applicable for the model reduction of 
    many models in kinetic theory.
    \newline

    \newline

  \end{abstract}

\section{Introduction}
Kinetic equations \cite{Villanireview} have wide-ranging applications 
in various areas of science and engineering, such as in the study of 
rarefied gas dynamics \cite{Grad}, plasma physics \cite{Balescu}, 
radiation astronomy \cite{mihalas}, 
neuroscience \cite{cai2004effective}, 
and population dynamics \cite{ramkrishna2000population}. 
These equations describe the evolution of the distribution functions 
of a large number of particles or agents interacting with each other 
through various mechanisms. 
However, their high-dimensional nature can pose significant 
computational challenges in solving them.
Therefore, it is often necessary to reduce their complexity to gain 
insight into the underlying physical phenomena and to perform 
numerical simulations. 

Numerous approaches have been developed for the simplification, 
reduction, and numerical approximation of kinetic equations. 
For instance, in the context of the 
Boltzmann equation \cite{Boltzmann,Cercignani}, 
a range of models, including 
Bhatnagar--Gross--Krook (BGK) \cite{BGK}, 
Shakhov \cite{Shakhov}, and 
Ellipsoidal BGK (ES-BGK) models \cite{Holway}, have been 
proposed to simplify the complicated collision term. Moreover, many  
reduced models and numerical methods 
for kinetic equations are based on an ansatz
or a priori knowledge of the distribution functions.
For instance, Levermore's moment models \cite{Levermore} suppose
that the distribution functions minimize the entropy under the moment 
constraints, Hermite spectral methods 
\cite{Hu2020Numerical,QuadraticCol} employ an ansatz characterized by 
a global Maxwellian multiplied by a polynomial, while globally 
hyperbolic moment equations (HME) \cite{Fan_new}, roughly speaking, 
utilize an ansatz of a local Maxwellian times a polynomial. 
Discrete velocity methods (DVM) rely on an ansatz supported on 
discrete microscopic velocity points, 
quadrature methods of moments (QMOM) \cite{McGraw} employ a linear 
combination of Dirac delta functions, and 
extended quadrature methods of moments (EQMOM) 
\cite{yuan2012extended} 
use Gaussian functions. These methods 
underscore the multifaceted efforts to reduce  
kinetic equations based on an ansatz of the distribution functions, 
setting the stage for exploring model reduction in this 
paper.

While the kinetic equations, and indeed many other 
problems, are formulated in high-dimensional
(resp. infinite-dimensional) spaces, the 
solutions to these equations often lie on 
low-dimensional (resp. finite-dimensional) 
manifolds. Firstly, the problems of interest may 
possess initial and boundary data on 
low-dimensional manifolds. Secondly, the potential dissipative nature 
of these equations tends to attract solutions toward low-dimensional 
manifolds \cite{Carr1981,jin2022}. 
Consequently, what appears as a
high-dimensional problem can possess essentially
low-dimensional characteristics. Furthermore, in practical terms, 
achieving a complete and detailed resolution of a high-dimensional 
function without any a priori knowledge is a formidable challenge in 
numerical simulations due to the notoriously difficult problem known 
as the curse of dimensionality.
Therefore, a promising approach to reducing the complexity of kinetic 
equations is to identify these 
intrinsically low-dimensional features, accompanied by the 
development of 
corresponding numerical methods and model reduction techniques
\cite{antoulas,Benner}.

Unfortunately, acquiring the 
structure of the inherently low-dimensional solution manifolds 
a priori is challenging. 
Nevertheless, 
it is feasible to construct low-dimensional approximate solution 
manifolds that can serve as an ansatz for the solutions. Such an 
ansatz can be formulated through several approaches, such as 
mathematically straightforward formulas and a priori 
knowledge based on physical insight. 
In particular, machine learning techniques provide effective 
low-dimensional expressions of high-dimensional functions
\cite{ghojogh2023elements,HanE2018,Lee2007,vidal2016}, 
which can be captured by learning the data of 
solutions generated from a specific low-dimensional set of input data.
Our primary focus in this 
study centers on developing reduced models, which preserve specific 
structural properties based on the given ansatz manifolds.

This paper introduces a novel framework for model reduction in the 
context of general kinetic equations from a geometric point of view
\cite{Lang1995}. Our approach involves a 
fundamental procedure: the projection from the tangent bundle 
of the solution manifold onto the tangent bundle of the ansatz 
manifold. To illustrate this concept, consider the following model 
problem without the convection term:
\begin{displaymath}
  \od{u}{t} = Q |_{u} \in T_u M, \quad u(0) = u_0 \in M,
\end{displaymath}
where $u : t \in \bR \mapsto u(t) \in M$ is the unknown function,
$M$ represents the solution manifold, which may be either finite- or 
infinite-dimensional, $T M$ denotes the tangent bundle of 
the manifold $M$, and $Q \in \Gamma(T M) : M \to T M$ 
represents a tangent vector 
field on $M$. To reduce this system, we introduce an ansatz manifold 
$\hat{M}$ as an embedded submanifold of $M$ via the embedding 
$i : \hat{M} \to M$. The objective is to confine the solution to the 
ansatz manifold $\hat{M}$, and to obtain the approximate solution
$\hat{u} : t \in \bR \mapsto \hat{u}(t) \in \hat{M}$.
Two issues arise: firstly, the 
initial value $u_0$ may not generally lie on $i(\hat{M})$; secondly,  
even if a solution $i(\hat{u})$ on the ansatz manifold 
$i(\hat{M})$ is obtained at some point, 
the right-hand term $Q|_{i(\hat{u})}$ may not 
necessarily belong to the space $i_* T_{\hat{u}} \hat{M}$, where 
$i_* : T \hat{M} \to T M$ is the pushforward of the embedding $i$. 
To address these problems, we need to introduce a projection 
$p : M \to \hat{M}$, such that $p \circ i$ is the identity
on $\hat{M}$, to project the initial value onto the ansatz manifold 
$\hat{M}$. 
The projection $p$ induces a pushforward 
$p_* : T M \to T \hat{M}$, which can project the right-hand term 
onto the tangent space $T_{\hat{u}} \hat{M}$ of the ansatz manifold 
$\hat{M}$. 
The reduced model can be presented as 
follows:
\begin{displaymath}
  \od{\hat{u}}{t} = \hat{Q} |_{\hat{u}} \in T_{\hat{u}} \hat{M}, 
  \quad \hat{u}(0) = p(u_0) \in \hat{M}, 
\end{displaymath}
where $\hat{Q} = p_* \circ Q \circ i \in \Gamma(T \hat{M}) : 
\hat{M} \to T \hat{M}$ 
represents a tangent vector field on $\hat{M}$.
Alternatively, when examining the dynamics on the solution manifold 
$M$, 
the reduced model can be represented as:
\begin{displaymath}
  \od{i(\hat{u})}{t} = i_* p_* Q|_{i(\hat{u})} \in T_{i(\hat{u})} M, 
  \quad i(\hat{u}(0)) = (i \circ p)(u_0) \in M.
\end{displaymath}
Here $i \circ p$ and $i_* p_*$ correspond to the (possibly nonlinear) 
projection from the manifold $M$ to its submanifold $\hat{M}$, and 
the linear projection from the tangent space 
$T_{i(\hat{u})} M$ onto its subspace $i_* T_{\hat{u}} \hat{M}$, 
respectively, with the property that $(i \circ p)^2 = i \circ p$ and 
$(i_* p_*)^2 = i_* p_*$. 

When it comes to the model reduction for kinetic equations, we 
project both the 
convection and collision terms onto the tangent space of the ansatz 
manifold. 
The reduced systems obtained using our framework are
first-order partial differential equations (PDEs) with the same 
structure as the kinetic equations.
Here, we only require $i_* p_*$ to be a projection between 
the tangent bundle but have yet to specify 
with respect to which inner product the 
projection is orthogonal. Once each point 
$u$ on the smooth manifold $M$ is associated with an inner product 
$g|_u$ on the tangent space $T_u M$ that varies smoothly with $u$, a 
Riemannian manifold $(M, g)$ is 
obtained. We find that the choice of the Riemannian metric 
$g$ depends on the hyperbolicity, which is crucial to ensure the 
well-posedness of the equations. Furthermore, we 
express several other crucial structural properties  
in geometric terms within the unified framework for kinetic 
equations and reduced models, 
including conservation laws, entropy dissipation, 
finite propagation speed, and linear stability conditions. 
We prove that model reduction preserves 
these structural properties under certain conditions on the ansatz 
manifold and the Riemannian metric.

With the structure-preserving properties of model reduction, studying 
the properties of the reduced models reduces to studying the kinetic 
equations. In particular, we investigate when kinetic equations 
possess properties of generalized hyperbolicity and linear stability, 
and thus, reduced models can inherit these properties. We find that 
these properties are related to the choice of the Riemannian metric. 
We also study conservation laws and entropy dissipation in the 
context of kinetic equations.

Firstly, we give an equivalent characterization of the 
Riemannian metric that makes the kinetic equations hyperbolic. A 
necessary and sufficient condition is that the Riemannian metric is 
an $L^2$ inner product, which means that the Riemannian metric 
involved in the model reduction should not contain any information 
about the derivatives w.r.t. the ordinate variables.

Secondly, we discuss rigorously the 
relation between the H-theorem of the kinetic equations and the 
structural stability conditions of the reduced models for the first 
time. The structural stability conditions proposed by 
\cite{yong1999singular} are another crucial property of the reduced 
models. They characterize the dissipation properties and ensure 
well-posedness under zero relaxation limits. It is widely believed 
that these conditions for the reduced models are a proper counterpart 
of the H-theorem for the kinetic equations 
\cite{Di2017stability,Huang2022stability,Huang2020stability,
zhang2023stability,Zhao2017stability}, and numerous works aim to 
verify these conditions for various reduced models. However, there 
has never been a rigorous discussion on the relation between 
these two properties in the literature. Applying the novel framework, 
we can prove that, under a natural requirement on the Riemannian 
metric determined by the entropy, a non-degenerate version of 
H-theorem (resp. H-theorem) for kinetic equations yields linear 
stability conditions (resp. weak linear stability conditions), which 
the model reduction can preserve.
We also prove a converse result that a stronger version of linear 
stability conditions can yield the H-theorem in some sense.

Thirdly, we give an equivalent and verifiable characterization to 
determine whether a quantity has a flux. 
Roughly speaking, the meaning of this characterization is that the 
Hessian operator of this quantity in the ordinate representation 
is analogous to a diagonal matrix. 
This result can help to rule 
out a specific quantity as a candidate for a conserved quantity or 
an entropy when studying a new kinetic equation.

The rest of this paper is structured as follows. 
\Cref{sec:framework} establishes an abstract framework for 
first-order PDEs that take values on 
Riemannian manifolds, unifying discussions on kinetic equations and 
the reduced models. 
\Cref{sec:reduction} presents our natural model reduction framework 
for these PDEs, including the projection onto tangent spaces, 
structure-preserving properties, and explicit coordinate 
representations of reduced models. 
\Cref{sec:kinetic} applies our framework to kinetic equations in a 
general form, which, in particular, provides equivalent 
characterizations to ensure hyperbolicity and quantities with fluxes, 
a rigorous discussion on the relation between H-theorem and 
structural stability conditions, and 
a posteriori error estimate. 
Several examples are presented to illustrate our framework. 
Finally, the paper ends with a summary and concluding remarks 
in \Cref{sec:concl}.

\section{The unified framework}
\label{sec:framework}
This section presents an abstract theory of our framework, 
which is suitable for our novel model reduction method.
To maximize the generality of our statement, we consider first-order 
PDEs that take values on a Riemannian manifold $(M, g)$,
with source terms, where the manifold $M$ may 
be either finite- or infinite-dimensional, and the Riemannian metric 
$g$ is related to the hyperbolicity and linear stability. 
This methodology, which integrates kinetic equations and 
traditional first-order PDEs into a unified framework, enables a 
comprehensive and systematic approach to studying the 
structure-preserving properties of model reduction methods.

\subsection{First-order PDEs taking values on a Riemannian manifold}
Consider the first-order PDE taking values on a 
Riemannian manifold $(M, g)$ as follows,
\begin{equation}\label{eq:1stonM}
  \begin{aligned}
    & \pd{u}{t} + A^j \pd{u}{x^j} = Q |_{u}, \\
    & u(0, \cdot) = u_0.
  \end{aligned}
\end{equation}
Here, Einstein notation is adopted. 
The time variable is $t \in J$, where $J$ is an open interval 
on $\bR$ containing $0$, 
the spatial coordinates are $\+x = (x^1, x^2, \ldots, x^d) \in \bR^d$,
and $u : J \times \bR^d \to M$ is an unknown function of $(t, \+x)$ 
taking values on the \textit{solution manifold} $M$,
which may be either finite- or infinite-dimensional.
We denote by $A^j \in \Gamma(T^{1,1} M) : T M \to T M$ 
a $(1,1)$-tensor field on $M$ for each $j = 1, 2, \ldots, d$, and by
$Q \in \Gamma(T M) : M \to T M$ a tangent vector field on $M$.
Roughly speaking, when the manifold $M$ is finite-dimensional, 
\cref{eq:1stonM} represents a first-order PDE in the classical 
sense, whereas, in the infinite-dimensional case, 
it represents kinetic 
equations such as the Boltzmann equation. 
We will demonstrate how the framework can be applied to 
the kinetic equations in \Cref{sec:kinetic}.

In many practical cases, the spatial variable is  
$\+x \in \Omega$, where $\Omega$ is a domain in $\bR^d$.
This paper focuses exclusively on the model reduction for Cauchy 
problems, i.e., when $\Omega = \bR^d$.

\subsection{Structural properties}
\label{sec:framework_structure}
We will discuss several structural properties of 
\cref{eq:1stonM}, including hyperbolicity, conservation laws, 
entropy dissipation, finite propagation speed, and linear stability.

\subsubsection{Hyperbolicity}
When the manifold $M$ is finite-dimensional, \cref{eq:1stonM} reduces 
to a first-order quasi-linear PDE in the 
classical sense. For such equations, the concept of 
hyperbolicity plays a crucial role in ensuring the well-posedness. 
For symmetric hyperbolic systems, there are abundant results on the 
local existence and uniqueness of the solution for both initial value 
and initial-boundary value problems 
\cite{Chen, Kato, Secchi, Trebeschi}. 
We now propose a new definition of hyperbolicity for 
\cref{eq:1stonM} w.r.t. the Riemannian metric $g$, which 
generalizes the definition of symmetrizable hyperbolic systems and 
is  
suitable for studying model reduction methods for kinetic equations. 

\begin{definition}[Hyperbolicity w.r.t. a Riemannian metric]
  \label{def:hyperbolicity}
  We say that \cref{eq:1stonM} is \textit{hyperbolic} w.r.t. the 
  Riemannian metric $g$ if the $(1, 1)$-tensor field $A^j$ is 
  symmetric w.r.t. $g$, i.e.,
  \begin{equation}\label{eq:hyperbolic}
    g(A^j X, Y) = g(X, A^j Y), 
  \end{equation}
  for each $X, Y \in \Gamma(T M)$ and each $j = 1, 2, \ldots, d$.
\end{definition}

The novel definition of hyperbolicity generalizes the conventional 
definition of first-order symmetrizable hyperbolic systems, as 
illustrated by the following example.

\begin{example}
  Suppose that the manifold $M$ is an open set in $\bR^n$ and that 
  \cref{eq:1stonM} is \textit{hyperbolic} w.r.t. the Riemannian 
  metric $g$.
  For each $u \in M$ and each $j = 1, 2, \ldots, d$, the operator 
  $A^j |_{u}$ is an $n \times n$ matrix. By definition, there exists 
  a $(1, 1)$-tensor field $A^0 \in \Gamma(T^{1,1} M)$ such that 
  $A^0 |_{u}$ is a symmetric and positive definite $n \times n$ 
  matrix for each $u \in M$ and that 
  $g(X, Y) |_{u} = \inn{A^0 X, Y} |_{u}$
  for any $X, Y \in \Gamma(T M)$. Thus, the hyperbolicity of 
  \cref{eq:1stonM} w.r.t. $g$ is equivalent to 
  that $A^0 A^j = (A^j)^\top A^0$ for each $j = 1, 2, \ldots, d$.
  This statement is exactly the conventional definition of first-order
  symmetrizable hyperbolic systems \cite{evans}, 
  and here the matrix $A^0 |_u$ 
  is the \textit{symmetrizer} of \cref{eq:1stonM}.
\end{example}

\begin{remark}\label{remark:hyperbolicity_dense}
  When the tangent space $T_u M$ is an infinite-dimensional Hilbert 
  space equipped with 
  the inner product $g|_u$ for each $u \in M$, the operator 
  $A^j |_u : T_u M \to T_u M$ is typically unbounded and densely 
  defined. In this case, when we refer to $A^j X|_u$ where 
  $X|_u \in T_u M$ is the tangent vector, as in \cref{eq:hyperbolic}, 
  it should be agreed that the tangent vector $X|_u$ lies in the 
  domain $D(A^j|_u)$ 
  of the operator $A^j|_u$.
\end{remark}

\subsubsection{Conservation laws and entropy dissipation}
\label{sec:framework_flux}
Conserved quantities and entropy are among the significant concerns
in kinetic theory \cite{Villanireview}. 
It is also important to consider whether these 
properties are preserved in reduced models. 
The first requirement to make a quantity either conserved or monotone
is that such a quantity has a flux.
In light of this, let us first introduce the 
concept of flux for \cref{eq:1stonM}.

\begin{definition}[Flux]\label{def:flux}
  Let $c : M \to \bR$ and $\+F_{c} = 
  (F_{c}^1, F_{c}^2\ldots, F_{c}^d) : M \to \bR^d$
  be smooth functions.
  We say that $\+F_{c}$ is the \textit{flux} of $c$ for 
  \cref{eq:1stonM} if the following equality holds,
  \begin{displaymath}
    \rd F_{c}^j = C (\rd c \otimes A^j) \in \Gamma(T^* M),
  \end{displaymath}
  for each $j = 1, 2, \ldots, d$, where $C$ denotes the tensor 
  contraction.
\end{definition}

\Cref{def:flux} is equivalent to the statement that 
$\inn{\rd F_c^j, X} = \inn{\rd c, A^j X}$ 
for each $X \in \Gamma(T M)$.
We define conserved quantities and entropy for \cref{eq:1stonM} 
as follows.

\begin{definition}[Conserved quantities and entropy]
  \label{def:cqe}
  Let $\+F_{c}$ be the flux of $c$ for \cref{eq:1stonM}.
  We say that $c$ is a \textit{conserved quantity} if 
  $\inn{\rd c, Q} = 0$; 
  and that $c$ is an \textit{entropy} if $\inn{\rd c, Q} \le 0$.
\end{definition}

\begin{remark}\label{rmk:flux}
  If $\+F_{c}$ is the flux of $c$ for \cref{eq:1stonM}, then
  \begin{displaymath}
    \pd{c}{t} + \nabla_{\+x} \cdot \+F_{c} = 
      \inn{\rd c, \pd{u}{t}} + 
      \inn{\rd F_{c}^j, \pd{u}{x^j}} = 
      \inn{\rd c, \pd{u}{t} + A^j \pd{u}{x^j}} = 
      \inn{\rd c, Q} |_{u},
  \end{displaymath}
  where $c \circ u$ and $\+F_{c} \circ u$ are abbreviated as 
  $c$ and $\+F_{c}$, respectively.
  By integrating over $\bR^d$, one obtains that 
  \begin{displaymath}
    \od{}{t} \int_{\bR^d} c \,\rd \+x = 
      \int_{\bR^d} \inn{\rd c, Q} |_{u} \,\rd \+x,
  \end{displaymath}
  assuming the sufficiently fast decay of $\+F_c \circ u$ for the 
  solution $u$ to \cref{eq:1stonM} at infinity.
  Therefore, if $c$ is a conserved quantity or an entropy, 
  the quantity $\int_{\bR^d} c \,\rd \+x$ is conserved or 
  monotonically decreasing, respectively.
  In addition, if there exists a smooth injective function 
  $\+c : M \to \bR^n$
  such that each of its components is a conserved quantity, then 
  \cref{eq:1stonM} is a \textit{hyperbolic system of balanced laws}.
\end{remark}

\begin{remark}
  Our current discussion focuses not strictly on entropy but 
  rather on monotone quantities since \Cref{def:cqe} does not
  consider an essential property of entropy, namely convexity, 
  which is crucial to establishing linear stability conditions. 
  We intend to explore this issue further in 
  \Cref{sec:framework_linear,sec:kinetic_linear}.
\end{remark}

\subsubsection{Finite propagation speed}
Certain kinetic equations, particularly those incorporating 
relativistic effects, exhibit a notable property that their 
propagation speed is finite. 
Examples of such equations include the radiative transfer equation
\cite{chandrasekhar} and the relativistic Boltzmann equation
\cite{relativistic}, both of which uphold the 
constraint of the propagation speed remaining within the 
speed of light. 
When constructing reduced models, it is desirable that 
the propagation speed of these models also respects this condition. 

Let us introduce the property of finite 
propagation speed for \cref{eq:1stonM}.
Suppose that the tangent space $T_u M$ constitutes a Hilbert space 
equipped with the inner product $g|_{u}$ for each $u \in M$ and that
\cref{eq:1stonM} is hyperbolic w.r.t. the Riemannian metric $g$.
These assumptions ensure the symmetry of the operator 
$A^j |_{u} : T_u M \to T_u M$ for each $j = 1, 2, \ldots, d$. 
For the sake of simplicity, suppose that the operator $A^j |_{u}$ is 
self-adjoint. Therefore, the spectrum of the operator 
$\+\sigma \cdot \+A |_{u} := \sigma_j A^j |_{u}$ is confined to a 
subset of $\bR$ for each 
$\+\sigma = (\sigma_j)_{j = 1}^d \in \bS^{d - 1} \subset \bR^d$. 
Denote the spectrum radius of $\+\sigma \cdot \+A |_{u}$ by 
$\rho(\+\sigma \cdot \+A |_{u})$. 

\begin{definition}[Finite propagation speed]
  The \textit{maximum propagation speed} of \cref{eq:1stonM} at 
  the point $u \in M$ is defined as $\rho(\+A|_{u}) := 
  \sup_{\+\sigma \in \bS^{d - 1}} \rho(\+\sigma \cdot \+A |_{u})$. 
  We say that \cref{eq:1stonM} has \textit{finite propagation speed}
  at $u \in M$ if $\rho(\+A|_{u}) < +\infty$.
\end{definition}

\begin{remark}
  If the manifold $M$ is finite-dimensional, the propagation 
  speed is always finite at a fixed point $u \in M$.
  A good reduced model is expected to have a propagation speed that 
  does not exceed the propagation speed of the original kinetic 
  equations.
\end{remark}

\subsubsection{Linear stability}
\label{sec:framework_linear}
Structural stability conditions \cite{yong1999singular,yong2001basic} 
play a fundamental role in establishing the linear stability of the 
equilibrium \cite{Cercignani} and ensuring well-posedness under the 
relaxation limit \cite{yong1999singular} for first-order hyperbolic 
relaxation systems. These conditions serve as the counterpart of the 
H-theorem for kinetic equations.
Let us introduce several new linear stability 
conditions for \cref{eq:1stonM} within the abstract framework 
to investigate the property of model reduction on preserving linear 
stability conditions and to discuss the relation between these  
conditions and H-theorem.
These conditions extend the scope of structural stability conditions 
originally designed for first-order hyperbolic relaxation systems.

The equilibrium state of \cref{eq:1stonM} is particularly interesting 
as it characterizes the continuum limit of \cref{eq:1stonM}. 
Before presenting the new linear stability conditions, we need to 
make some necessary assumptions on the equilibrium state of 
\cref{eq:1stonM}.

\begin{assumption}
  \label{assump:equilibrium}
  \Cref{eq:1stonM} satisfies the following assumptions:
  \begin{itemize}
    \item The equilibrium set $\tilde{M}$ defined by 
    \begin{equation}\label{eq:tildeM}
      \tilde{M} := \left\{ u \in M \,\big|\, 
        Q |_{u} = 0 \in T_u M \right\}
    \end{equation}
    is a smooth embedded submanifold of the manifold $M$. 
    We call the manifold $\tilde{M}$ the 
    \textit{equilibrium manifold}.
    The embedding mapping here is denoted by 
    $\iota : \tilde{M} \rightarrow M$.
    \item There exists a projection $\pi : U \rightarrow \tilde{M}$,
    which is defined on an open set $U \subset M$ containing 
    the equilibrium manifold $\tilde{M}$,
    such that $\pi \circ \iota = \id_{\tilde{M}}$ and that 
    \begin{equation}\label{eq:proj_require}
      g(X, \iota_* \tilde{Y}) |_{\iota(\tilde{u})} = 
        \tilde{g}(\pi_* X, \tilde{Y}) |_{\tilde{u}},
    \end{equation}
    for each $\tilde{u} \in \tilde{M}$, $X \in \Gamma(T M)$ and
    $\tilde{Y} \in \Gamma(T \tilde{M})$, 
    where the Riemannian metric $\tilde{g} := \iota^* g$ on 
    $\tilde{M}$ is defined as the pullback of $g$, i.e.,
    \begin{displaymath}
      \tilde{g}(\tilde{X}, \tilde{Y}) := 
      g(\iota_* \tilde{X}, \iota_* \tilde{Y}), 
    \end{displaymath}
    for each $\tilde{X}, \tilde{Y} \in \Gamma(T \tilde{M})$.
  \end{itemize}
\end{assumption}

\begin{example}
  To illustrate the meaning of \Cref{assump:equilibrium}, 
  let us consider the Boltzmann equation with two-body 
  collision term,
  \begin{equation}\label{eq:boltzmann}
    \pd{f}{t} + \+\xi \cdot \nabla_{\+x} f = Q[f], 
  \end{equation}
  where $\+\xi \in \+\Xi = \bR^d$ with $d = 3$,
  \begin{equation}\label{eq:boltzmann_twobody}
    Q[f] = \int_{\bR^d} \int_{S^{d-1}} 
    B(|\+\xi - \+\xi_*|, \cos \chi) 
    \big( f(\+\xi') f(\+\xi_*') - f(\+\xi) f(\+\xi_*) \big)
    \,\rd \+\sigma \,\rd \+\xi_*,
  \end{equation}
  the function $B = B(|\+\xi - \+\xi_*|, \cos \chi)$ is the Boltzmann 
  collision kernel, and
  \begin{equation}\label{eq:collision}
    \+\xi' := \frac{\+\xi + \+\xi_*}{2} + \frac{|\+\xi - \+\xi_*|}{2} 
    \+\sigma, \quad \+\xi_*' := 
    \frac{\+\xi + \+\xi_*}{2} - \frac{|\+\xi - \+\xi_*|}{2} \+\sigma, 
    \quad 
    \cos \chi := \+\sigma \cdot 
    \frac{\+\xi - \+\xi_*}{|\+\xi - \+\xi_*|}.
  \end{equation}
  The correspondence between the Boltzmann equation 
  \eqref{eq:boltzmann} and \cref{eq:1stonM} in the 
  abstract form is discussed in \Cref{sec:kinetic_general}.
  The submanifold $\tilde{M}$ in the first assumption is a 
  generalization of the manifold of Maxwellian distributions, which  
  represent the equilibrium states and play a crucial role in 
  understanding the macroscopic behavior of a gas. 
  The Maxwellian $f \in \tilde{M}$ can be expressed as 
  follows,
  \begin{displaymath}
    f(\+\xi) = \frac{\rho}{(2 \pi \theta)^{d / 2}} 
    \exp \left( - \frac{|\+\xi - \+u|^2}{2 \theta} \right),
  \end{displaymath}
  which is determined by several macroscopic quantities, 
  including 
  \begin{displaymath}
    \begin{aligned}
      \text{density: } \quad & 
      \rho := \int_{\bR^d} f(\+\xi) \,\rd \+\xi, \\
      \text{flow velocity: } \quad & 
      \+u := \frac{1}{\rho} \int_{\bR^d} \+\xi f(\+\xi) 
      \,\rd \+\xi, \\
      \text{temperature: } \quad & 
      \theta := \frac{1}{d \rho} \int_{\bR^d} |\+\xi - \+u|^2 
      f(\+\xi) \,\rd \+\xi.
    \end{aligned}
  \end{displaymath}
  It is natural to define a projection from the general distribution 
  function to the Maxwellian that preserves these macroscopic 
  quantities. 
  The projection $\pi$ introduced in the second assumption 
  generalizes this idea.
  We will investigate this point in detail later 
  in \Cref{thm:HassumpEqui}.
\end{example}

\begin{remark}\label{rmk:proj}
  According to the first item of \Cref{assump:equilibrium},
  the mapping $\iota$ yields a diffeomorphism 
  between $\tilde{M}$ and $\iota(\tilde{M})$ and an injective 
  pushforward mapping $\iota_* |_{\tilde{u}} : 
  T_{\tilde{u}} \tilde{M} \to T_{\iota(\tilde{u})} M$ at each point 
  $\tilde{u} \in \tilde{M}$.
  \Cref{eq:proj_require} in the second item of 
  \Cref{assump:equilibrium} is equivalent to 
  \begin{displaymath}
    g(X - \iota_* \pi_* X, \iota_* \pi_* Y) |_{\iota(\tilde{u})} = 0,
  \end{displaymath}
  for each $\tilde{u} \in \tilde{M}$ and $X, Y \in \Gamma(T M)$,
  which is equivalent to the statement that the 
  mapping $\iota_* \pi_* |_{\iota(\tilde{u})}$ is an orthogonal 
  projection from $T_{\iota(\tilde{u})} M$ to 
  $\iota_* T_{\tilde{u}} \tilde{M}$
  w.r.t. the inner product $g |_{\iota(\tilde{u})}$.
  Given a Riemannian metric $g$ on the manifold $M$,
  a projection $\pi$ can be constructed as follows,
  \begin{equation}\label{eq:RProj}
    \pi(u) \in \mathop{\arg \min}_{\tilde{u} \in \tilde{M}} \,
    \dist_g(u, \iota(\tilde{u})), \quad u \in M,
  \end{equation}
  which is uniquely determined and smooth near each 
  $\iota(\tilde{u}) \in M$ for each $\tilde{u} \in \tilde{M}$. Here, 
  $\dist_g$ is the distance on $M$ induced by the Riemannian metric 
  $g$. However, the projections $\pi$ satisfying 
  \cref{eq:proj_require} are not necessarily 
  unique; see \Cref{rmk:diff_g}.
\end{remark}

We can now present our definitions of generalized 
stability conditions, which will be shown later as generalizations 
of Yong's structural stability conditions 
\cite{yong1999singular,yong2001basic}.

\begin{definition}
  We say that \cref{eq:1stonM} satisfies the 
  \textit{generalized stability condition} (GSC) on $(M, g)$ 
  if it satisfies the following requirements:
  \begin{enumerate}[label=(\roman*),ref=(\roman*)]
    \item \label{item:stability_1}
    For each $\tilde{u} \in \tilde{M}$ and $X \in \Gamma(T M)$, 
    \begin{align}
      \label{eq:stability_1_1}
      D_{\iota_* \pi_* X} Q |_{\iota(\tilde{u})} = 0, \\ 
      \label{eq:stability_1_2}
      \iota_* \pi_* D_X Q |_{\iota(\tilde{u})} = 0,
    \end{align}
    where $D$ is the Levi-Civita connection on $(M, g)$;
    \item \label{item:stability_2}
    For each $X, Y \in \Gamma(T M)$ and $j = 1, 2, \ldots, d$,
    \begin{displaymath}
      g(A^j X, Y) = g(X, A^j Y);
    \end{displaymath}
    \item \label{item:stability_3}
    For each 
    $\tilde{u} \in \tilde{M}$ and $X \in \Gamma(T M)$ satisfying 
    that $X |_{\iota(\tilde{u})} \ne 0$ and 
    $\iota_* \pi_* X |_{\iota(\tilde{u})} = 0$,
    \begin{displaymath}
      g(D_X Q, X) |_{\iota(\tilde{u})} < 0.
    \end{displaymath}
  \end{enumerate}
\end{definition}

\begin{definition}
  We say that \cref{eq:1stonM} satisfies the 
  \textit{generalized uniform stability condition} (GUSC) 
  on $(M, g)$ if it satisfies 
  \ref{item:stability_1}, \ref{item:stability_2}, and 
  \begin{enumerate}[label=(\roman*),ref=(\roman*)]
    \setcounter{enumi}{3}
    \item \label{item:uniform_stability}
    For each $\tilde{u} \in \tilde{M}$ and each $X \in \Gamma(T M)$,
    there exists a constant $\lambda > 0$ such that
    \begin{equation}\label{eq:stability_3}
      g(D_X Q, X) |_{\iota(\tilde{u})} \le - \lambda 
      \norm{X - \iota_* \pi_* X}_g^2 \big|_{\iota(\tilde{u})}.
    \end{equation}
  \end{enumerate}
\end{definition}

\begin{definition}
  We say that \cref{eq:1stonM} satisfies the 
  \textit{generalized weak stability condition} (GWSC) on $(M, g)$
  if it satisfies \ref{item:stability_1}, \ref{item:stability_2}. and 
  \begin{enumerate}[label=(\roman*),ref=(\roman*)]
    \setcounter{enumi}{4}
    \item \label{item:weak_stability}
    for each $\tilde{u} \in \tilde{M}$ and each $X \in \Gamma(T M)$, 
    \begin{displaymath}
      g(D_X Q, X) |_{\iota(\tilde{u})} \le 0.
    \end{displaymath}
  \end{enumerate}
\end{definition}

\begin{remark}\label{rmk:stability_1_1}
  \Cref{eq:stability_1_1} is unconditionally satisfied 
  thanks to \Cref{assump:equilibrium}.
\end{remark}

\begin{remark}\label{rmk:diff_g}
  Here is a simple fact: for each $X, Y \in \Gamma(T M)$, if the 
  tangent vector $Y|_u = 0$ for some $u \in M$, then the covariant 
  derivative $D_X Y |_u$ is independent of the choice of the 
  Riemannian metric $g$. 
  Given a function 
  $F \in C^\infty(M)$ satisfying that $F(u) > 0$ for each $u \in M$, 
  define a new Riemannian metric $g_F$ as follows:
  \begin{displaymath}
    g_F(X, Y) |_u := F(u) g(X, Y) |_u,
  \end{displaymath}
  for each $u \in M$ and $X, Y \in \Gamma(T M)$. If 
  \cref{eq:1stonM} satisfies GSC (resp. GUSC or GWSC) on $(M, g)$,
  then it also satisfies GSC (resp. GUSC or GWSC) on $(M, g_F)$.
  In addition, different Riemannian metrics $g$ on the manifold $M$ 
  induce different projections $\pi$ by \cref{eq:RProj}, 
  which leads to the non-uniqueness of projections $\pi$ in 
  \Cref{rmk:proj}.
\end{remark}

\begin{lemma}\label{lemma:finiteD}
  In general, the following relations hold:
  \begin{displaymath}
    \text{GUSC} \Longrightarrow \text{GSC} 
    \Longrightarrow \text{GWSC}.
  \end{displaymath}
  Furthermore, if the manifold $M$ is finite-dimensional, 
  then GSC is equivalent to GUSC.
\end{lemma}

\begin{proof}
  If \cref{eq:1stonM} satisfies GUSC on $(M, g)$, then 
  for each $\tilde{u} \in \tilde{M}$ and $X \in \Gamma(T M)$ 
  satisfying that $X |_{\iota(\tilde{u})} \ne 0$ and 
  $\iota_* \pi_* X |_{\iota(\tilde{u})} = 0$, one has that 
  \begin{displaymath}
    g(D_X Q, X) |_{\iota(\tilde{u})} \le - \lambda 
    \norm{X - \iota_* \pi_* X }_g^2 \big|_{\iota(\tilde{u})} = 
    - \lambda \norm{X}_g^2 \big|_{\iota(\tilde{u})} < 0.
  \end{displaymath}
  Therefore, \cref{eq:1stonM} satisfies GSC.

  Suppose that \cref{eq:1stonM} satisfies GSC on $(M, g)$.
  For each $\tilde{u} \in \tilde{M}$, 
  define the linear space $V$ as follows,
  \begin{displaymath}
    V := \left\{ X|_{\iota(\tilde{u})} \in T_{\iota(\tilde{u})} M 
    \,\big|\, \iota_* \pi_* X |_{\iota(\tilde{u})} = 0 \right\}.
  \end{displaymath}
  For each $X \in \Gamma(T M)$, one has that 
  $(X - \iota_* \pi_* X) |_{\iota(\tilde{u})} \in V$, and thus,
  \begin{displaymath}
    g(D_X Q, X) |_{\iota(\tilde{u})} = 
      g(D_{X - \iota_* \pi_* X} Q, X - \iota_* \pi_* X)
      |_{\iota(\tilde{u})} \le 0,
  \end{displaymath}
  thanks to \ref{item:stability_1}.
  Furthermore, if the manifold $M$ is finite-dimensional, then 
  the linear space $V$ is 
  finite-dimensional. Define the set 
  $S := \left\{ X|_{\iota(\tilde{u})} \in V \,\big|\, 
  \|{X|_{\iota(\tilde{u})}}\|_g = 1 \right\}$. 
  Note that $S$ is compact in $V$,
  and $g(D_X Q, X) |_{\iota(\tilde{u})} < 0$ for each 
  $X|_{\iota(\tilde{u})} \in S$. Thus, there exists 
  $\lambda > 0$ such that 
  \begin{displaymath}
    g(D_X Q, X) |_{\iota(\tilde{u})} \le - \lambda, 
  \end{displaymath}
  for each $X|_{\iota(\tilde{u})} \in S$, 
  which yields that 
  \begin{displaymath}
    g(D_X Q, X) |_{\iota(\tilde{u})} \le - \lambda 
      \norm{X}_g^2 |_{\iota(\tilde{u})}, 
  \end{displaymath}
  for each $X|_{\iota(\tilde{u})} \in V$.
  Therefore, for each $X \in \Gamma(T M)$, one has that
  \begin{displaymath}
    g(D_X Q, X) |_{\iota(\tilde{u})} = 
      g(D_{X - \iota_* \pi_* X} Q, X - \iota_* \pi_* X)
      |_{\iota(\tilde{u})} \le - \lambda 
      \norm{X - \iota_* \pi_* X}_g^2 |_{\iota(\tilde{u})},
  \end{displaymath}
  which yields that \cref{eq:1stonM} satisfies GUSC.
\end{proof}

\begin{example}\label{exp:Yong}
  Suppose that the manifold $M$ is an open set in a Euclidean space 
  $V := \bR^n$. 
  A tangent vector field on $M$ can be seen as a mapping from $M$ to 
  $V$. For each $u \in M$, there exists a symmetric and positive 
  definite matrix $A^0|_u$ such that 
  \begin{displaymath}
    g(X, Y) |_{u} = X^\top A^0 Y |_{u},
  \end{displaymath}
  for each $X, Y \in \Gamma(T M)$. Let $r := \dim \tilde{M}$,
  and define $B := \left( \begin{smallmatrix}
    0 & 0 \\ 0 & I_r
  \end{smallmatrix} \right) \in \bR^{n \times n}$.
  For each $\tilde{u} \in \tilde{M}$, the mapping 
  $\iota_* \pi_* |_{\iota(\tilde{u})}$ is an orthogonal projection 
  w.r.t. the inner product $g |_{\iota(\tilde{u})}$ 
  from $T_{\iota(\tilde{u})} M$ to 
  $\iota_* T_{\tilde{u}} \tilde{M}$. Therefore, for each 
  $u \in \iota(\tilde{M})$, there exists an invertible matrix $P|_u$ 
  such that 
  \begin{displaymath}
    \iota_* \pi_* X |_{u} = P^{-1} (I - B) P X |_{u},
  \end{displaymath}
  satisfying that $B P^{-\top} A^0 P^{-1} (I - B)|_u = 0$.
  Note that $Q |_{u} = 0$ for $u \in \iota(\tilde{M})$, which yields 
  that $D_X Q |_{u} = Q_u X |_{u}$, 
  where $Q_u$ is the gradient of the vector-valued function $Q$ 
  w.r.t. $u$.
  \Cref{item:stability_1} is equivalent to 
  \begin{displaymath}
    Q_u P^{-1} (I - B) P |_{u} = 0, \quad 
    P^{-1} (I - B) P Q_u |_{u} = 0,
  \end{displaymath}
  for each $u \in \iota(\tilde{M})$, 
  which is equivalent to the statement that there exists 
  $S|_u \in \bR^{r \times r}$ such that 
  \begin{displaymath}
    P Q_u P^{-1} |_{u} = \begin{pmatrix}
      0 & 0 \\ 0 & S
    \end{pmatrix} \bigg|_{u}.
  \end{displaymath}
  \Cref{item:stability_2} is equivalent to 
  \begin{displaymath}
    A^0 A^j |_u = (A^j)^\top A^0 |_u, 
  \end{displaymath}
  for each $u \in M$ and $j = 1, 2, \ldots, d$.
  \Cref{item:uniform_stability} is equivalent to
  \begin{displaymath}
    g(D_X Q, X) |_{u} = \frac{1}{2} X^\top (A^0 Q_u + Q_u^\top A^0) X 
    |_{u} \lesssim - \|{P^{-1} B P X}\|_g^2 \big|_u \lesssim 
    - X^\top P^\top B P X |_u,
  \end{displaymath}
  for each $u \in \iota(\tilde{M})$ and $X \in \Gamma(T M)$,
  which is equivalent to 
  \begin{displaymath}
    (A^0 Q_u + Q_u^\top A^0) |_u \lesssim - P^\top B P |_u.
  \end{displaymath}
  Similarly, \Cref{item:weak_stability} is equivalent to 
  \begin{displaymath}
    (A^0 Q_u + Q_u^\top A^0) |_u \le 0,
  \end{displaymath}
  for each $u \in \iota(\tilde{M})$.
\end{example}

We summarize the discussions in \Cref{exp:Yong} and obtain the 
following theorem. The theorem shows that our generalized stability
conditions generalize Yong's first and second stability conditions; 
see \cite[p. 277--278]{yong2001basic}. Note that the first stability 
condition we mention here is modified and is weaker than the original 
version in \cite{yong2001basic} since GWSC does not require 
the matrix $S|_u$ to be stable.

\begin{theorem}
  If $M$ is an open set in a Euclidean space, the following 
  relations hold:
  \begin{displaymath}
    \begin{aligned}
      & \text{Yong's first stability condition} 
        \Longleftrightarrow \text{GWSC}, \\
      & \text{Yong's second stability condition} 
        \Longleftrightarrow \text{GSC} 
        \Longleftrightarrow \text{GUSC}.
    \end{aligned}
  \end{displaymath}
\end{theorem}

  \section{Natural model reduction}
\label{sec:reduction}
This section presents our novel natural model reduction method and 
proves that the reduced models preserve several crucial physical 
properties under certain conditions.
We also give an explicit expression of the reduced models in 
coordinate form.

\subsection{Motivations and methodologies}
Once applied to a specific scenario, the kinetic equation is 
equipped with prescribed initial and boundary values. We aim to 
solve such a problem with input data in a low-dimensional 
configuration space. The solution operator of the equation maps each 
point in the input data space to a corresponding solution function. 
Therefore, the solution functions of the equation with all
the input data form a low-dimensional manifold in the solution
space, provided that the solution operator is continuous. This 
observation motivates us to develop a strategy for reducing the 
kinetic equations to low-dimensional models.

Although it is known that the solution for low-dimensional input data 
is confined to a low-dimensional manifold, the explicit expression 
of such a manifold is often unknown a priori, posing a significant 
challenge in this procedure. It is generally 
believed that the functions on such a solution manifold cannot be 
expressed in elementary functions, and obtaining its explicit 
expression can be elusive. Despite 
these difficulties, it is often possible to obtain an approximate 
solution manifold through physical intuition, 
a priori knowledge of the structure of the solution functions, or the 
need for mathematical simplicity. In particular, machine learning 
techniques can also be employed to approximate the 
solution manifold. 

For our present purposes, suppose that an approximate 
solution is obtained through an ansatz, which postulates that the 
approximate solution lies on a submanifold $\hat{M}$ embedded in the 
Riemannian manifold $(M, g)$.
We call the manifold $\hat{M}$ the \textit{ansatz manifold} or 
the \textit{approximate manifold}.
We denote the embedding as 
$i : \hat{M} \to M$, which induces a
pushforward $i_* : T \hat{M} \to T M$, 
and a Riemannian metric on $\hat{M}$ defined as 
$\hat{g} := i^* g \in \Gamma(T^{0,2} \hat{M})$.
Notably, the manifold $\hat{M}$ is typically finite-dimensional.

To obtain a reduced model of \cref{eq:1stonM} with the unknown 
$\hat{u}: \hat{J} \times \bR^d \to \hat{M}$, 
one may attempt to plug $i \circ \hat{u}$ into \cref{eq:1stonM} 
directly, resulting in  
\begin{equation}\label{eq:reduced_wrong}
  \begin{aligned}
    & i_* \pd{\hat{u}}{t} + A^j i_* \pd{\hat{u}}{x^j} = 
      Q |_{i(\hat{u})}, \\
    & (i \circ \hat{u}) (0, \cdot) = u_0.
  \end{aligned}
\end{equation}
However, note that $i_* \tfrac{\partial \hat{u}}{\partial t}$ and 
$(i \circ \hat{u}) (0, \cdot)$ 
take values in $i_*(T \hat{M})$ and $i(\hat{M})$, respectively, while 
$Q |_{i(\hat{u})} - A^j i_* \tfrac{\partial \hat{u}}{\partial x^j}$ 
and $u_0$ may not in general. 
To resolve this issue, we need a projection $p : M \to \hat{M}$
such that $p \circ i$ is the identity on $\hat{M}$, which 
yields a pushforward $p_* : T M \to T \hat{M}$ 
and a pullback $p^* : T^* \hat{M} \to T^* M$.
Applying $p_*$ and $p$ to both sides of the first and second lines 
of \cref{eq:reduced_wrong}, respectively, yields the reduced 
model given by
\begin{equation}\label{eq:reduced}
  \begin{aligned}
    & \pd{\hat{u}}{t} + \hat{A}^j \pd{\hat{u}}{x^j} = 
      \hat{Q} |_{\hat{u}}, \\
    & \hat{u}(0, \cdot) = p \circ u_0,
  \end{aligned}
\end{equation}
where $\hat{A}^j := p_* \circ A^j \circ i_* \in 
\Gamma(T^{1,1} \hat{M}) : T \hat{M} \to T \hat{M}$ is a 
$(1, 1)$-tensor field on $\hat{M}$, and 
$\hat{Q} := p_* \circ Q \circ i \in \Gamma(T \hat{M}) : 
\hat{M} \to T \hat{M}$ is 
a tangent vector field on $\hat{M}$.

\subsection{Structure-preserving properties}
\label{sec:reduction_preserve}

The reduced model \eqref{eq:reduced} has the same form as the 
original equation \eqref{eq:1stonM}, indicating that the model 
reduction preserves the structure of the equations, 
which is a benefit 
of incorporating Riemannian geometry into the model reduction method. 
We can now investigate the kinetic equations and the 
reduced models within a unified framework.
We can further explore the structure-preserving properties of the 
method, such as hyperbolicity, conservation laws, entropy 
dissipation, finite propagation speed, and linear stability.

\subsubsection{Hyperbolicity}
We begin by examining the preservation of hyperbolicity of the model 
reduction method, summarized by the following result.

\begin{theorem}\label{thm:reduce_hyperbolicity}
  If $i_*$ and $p_*$ are adjoint mutually, i.e., if
  \begin{equation}\label{eq:reduce_hyperbolicity}
    \hat{g} (p_* X, \hat{Y}) |_{\hat{u}} = 
      g(X, i_* \hat{Y}) |_{i(\hat{u})},
  \end{equation}
  holds for each $\hat{u} \in \hat{M}$, $X \in \Gamma(T M)$ and 
  $\hat{Y} \in \Gamma(T \hat{M})$, then the model reduction 
  preserves hyperbolicity.
\end{theorem}

\begin{proof}
  Suppose that \cref{eq:1stonM} is hyperbolic w.r.t. 
  the Riemannian metric $g$.
  Note that 
  \begin{displaymath}
      \hat{g} (p_* A^j i_* \hat{X}, \hat{Y}) = 
      g (A^j i_* \hat{X}, i_* \hat{Y}) = 
      g (i_* \hat{X}, A^j i_* \hat{Y}) = 
      \hat{g} (\hat{X}, p_* A^j i_* \hat{Y}),
  \end{displaymath}
  for each $\hat{X}, \hat{Y} \in \Gamma(T \hat{M})$ and each 
  $j = 1, 2, \ldots, d$. Therefore, \cref{eq:reduced} is hyperbolic 
  w.r.t. the Riemannian metric $\hat{g}$.
\end{proof}

\begin{remark}
  According to the definition of $\hat{g}$, 
  \cref{eq:reduce_hyperbolicity} is equivalent to the statement that 
  the mapping $i_* p_*$ is an orthogonal projection from 
  $T_{i(\hat{u})} M$ to $i_* T_{\hat{u}} \hat{M}$ w.r.t. the 
  inner product $g|_{i(\hat{u})}$, similar to the 
  discussions in \Cref{rmk:proj}.
\end{remark}

\begin{remark}
  Suppose that $T_u M$ is a Hilbert space with the inner 
  product $g|_u$. The proof of \Cref{thm:reduce_hyperbolicity} 
  presented above applies to bounded operators $A^j$. In the case 
  that $A^j$ is unbounded, an issue arises that the domain 
  $D(\+\sigma \cdot \+{\hat{A}}|_{\hat{u}})$ with 
  $\hat{u} \in \hat{M}$ 
  and $\+\sigma \in \bS^{d - 1}$ may not be dense in the space 
  $T_{\hat{u}} \hat{M}$ or may even contain only the zero element 
  under the assumption that the domain 
  $D(\+\sigma \cdot \+A|_u)$ is dense in the space $T_u M$ for each 
  $u \in M$ and each $\+\sigma \in \bS^{d - 1}$, as mentioned 
  in \Cref{remark:hyperbolicity_dense}. 
  To address this issue, we need to make additional assumptions on 
  the ansatz manifold $\hat{M}$. 
  Specifically, we require that $i_* T_{\hat{u}} \hat{M}$ is a closed 
  subspace of $T_{i(\hat{u})} M$ and that $D(\hat{A} |_{\hat{u}}) = 
  (i_*)^{-1} \big( D(A|_{i(\hat{u})}) \cap 
  i_* T_{\hat{u}} \hat{M} \big)$ is dense in the space 
  $T_{\hat{u}} \hat{M}$.
  In practical applications, these requirements are easily satisfied.
  These requirements hold when $\hat{M}$ is finite-dimensional, and 
  $i_* T_{\hat{u}} \hat{M}$ is a subset of $D(A|_{i(\hat{u})})$.
  From now on, we shall always make these assumptions. Furthermore, 
  at this time, the operator $\+\sigma \cdot \+{\hat{A}} |_{\hat{u}}$ 
  is also symmetric (resp. self-adjoint) if the operator 
  $\+\sigma \cdot \+A|_{i(\hat{u})}$ is symmetric 
  (resp. self-adjoint).
\end{remark}

\subsubsection{Conservation laws and entropy dissipation}
Before exploring the preservation of conserved quantities and 
entropy, let us first consider the preservation of fluxes. 
We state the following theorem.

\begin{theorem}\label{thm:reduce_flux}
  Suppose that $\+F_c$ is the flux of $c$ for 
  \cref{eq:1stonM}. Let $\hat{c} := c \circ i$ and 
  $\+{\hat{F}}_{\hat{c}} := \+F_c \circ i$ be naturally induced by 
  $c$ and $\+F_c$, respectively.
  If 
  \begin{equation}\label{eq:flux_proj}
    p^* \rd \hat{c} |_{\hat{u}} = \rd c |_{i(\hat{u})},
  \end{equation}
  for each 
  $\hat{u} \in \hat{M}$, then $\+{\hat{F}}_{\hat{c}}$ is the flux of 
  $\hat{c}$ for \cref{eq:reduced}.
  Furthermore, one has that 
  \begin{displaymath}
    \pd{\hat{c}}{t} + \nabla_{\+x} \cdot \+{\hat{F}}_{\hat{c}} = 
      \inn{\rd c, Q} |_{i(\hat{u})}.
  \end{displaymath}
\end{theorem}

\begin{proof}
  Note that 
  \begin{displaymath}
    \inns{\rd \hat{F}^j_{\hat{c}}, \hat{X}} |_{\hat{u}} = 
    \inns{\rd F^j_{c}, i_* \hat{X}} |_{i(\hat{u})} = 
    \inns{\rd c, A^j i_* \hat{X}} |_{i(\hat{u})} = 
    \inns{\rd \hat{c}, p_* A^j i_* \hat{X}} |_{\hat{u}} = 
    \inns{\rd \hat{c}, \hat{A}^j \hat{X}} |_{\hat{u}},
  \end{displaymath}
  for each $\hat{u} \in \hat{M}$, each 
  $\hat{X} \in \Gamma(T \hat{M})$ and each $j = 1, 2, \ldots, d$.
  Furthermore, one has that 
  \begin{displaymath}
    \inns{\rd \hat{c}, \hat{Q}} |_{\hat{u}} = 
    \inns{\rd \hat{c}, p_* Q i} |_{\hat{u}} = 
    \inns{\rd c, Q} |_{i(\hat{u})},
  \end{displaymath}
  for each $\hat{u} \in \hat{M}$.
  Therefore, by applying \Cref{rmk:flux}, we complete the proof of 
  this theorem.
\end{proof}

\begin{remark}\label{rmk:reduce_flux}
  \Cref{eq:flux_proj} is equivalent to the statement 
  that the equality 
  $\inn{\rd c, X} |_{i(\hat{u})} = 
  \inn{\rd c, i_* p_* X} |_{i(\hat{u})}$ holds 
  for each $X|_{i(\hat{u})} \in T_{i(\hat{u})} M$. 
  Furthermore, by Riesz representation theorem, 
  under the assumption of \cref{eq:reduce_hyperbolicity},
  which ensures the hyperbolicity of reduced models, 
  \cref{eq:flux_proj} is equivalent to the statement that there 
  exists a tangent vector 
  $\hat{X} |_{\hat{u}} \in T_{\hat{u}} \hat{M}$ such that 
  $\rd c |_{i(\hat{u})} = g(i_* \hat{X} |_{\hat{u}}, \cdot)$.
\end{remark}

\Cref{thm:reduce_flux} shows that the model 
reduction method preserves the conserved quantities and entropy 
as long as the condition \eqref{eq:flux_proj} holds for 
each $\hat{u} \in \hat{M}$.

\subsubsection{Finite propagation speed}
In this part, we demonstrate that the reduced models obtained 
through the natural model reduction preserve the property of finite 
propagation speed. Specifically, the 
maximum propagation speed does not increase
after the natural model reduction. We have the following 
theorem.

\begin{theorem}\label{thm:reduce_speed}
  Given the self-adjoint operator $A^j|_u$ in the Hilbert space 
  $(T_u M, g|_u)$ for each $u \in M$ and $j = 1, 2, \ldots, d$ 
  satisfying that $\rho(\+A|_u) < +\infty$, and a closed submanifold 
  $\hat{M}$ of $M$, If the condition \eqref{eq:reduce_hyperbolicity} 
  holds, then 
  \begin{displaymath}
    \rho(\+{\hat{A}}|_{\hat{u}}) \le \rho(\+A|_{i(\hat{u})}) 
    < +\infty.
  \end{displaymath}
\end{theorem}

\begin{proof}
  Note that $i_*(T_{\hat{u}} \hat{M})$ is a closed subspace of 
  $T_{i(\hat{u})} M$, and that 
  $i_* : T_{\hat{u}} \hat{M} \to i_*(T_{\hat{u}} \hat{M})$ is an 
  isomorphism that preserves the inner product for each 
  $\hat{u} \in \hat{M}$.
  Therefore, $T_{\hat{u}} \hat{M}$ is a Hilbert space with the inner 
  product $\hat{g}|_{\hat{u}}$.
  By the assumption on the spectral radius, one has that $A^j|_u$ is 
  a bounded operator for each $j = 1, 2, \ldots, d$. 
  By \Cref{thm:reduce_hyperbolicity}, one has that 
  $\hat{A}^j|_{\hat{u}}$ is also self-adjoint.
  Note that 
  \begin{displaymath}
    \begin{aligned}
      & \rho(\+{\hat{A}}|_{\hat{u}}) = \sup_{\+\sigma \in \bS^{d-1}}
      \sup_{\hat{X}|_{\hat{u}} \in T_{\hat{u}} \hat{M}} 
      \frac{\hat{g}(\+\sigma \cdot \+{\hat{A}} 
      \hat{X}, \hat{X})|_{\hat{u}}}
      {\hat{g}(\hat{X}, \hat{X})|_{\hat{u}}} = 
      \sup_{\+\sigma \in \bS^{d-1}}
      \sup_{\hat{X}|_{\hat{u}} \in T_{\hat{u}} \hat{M}} 
      \frac{g(\+\sigma \cdot \+A i_* \hat{X}, i_* \hat{X})|_{\hat{u}}}
      {g(i_* \hat{X}, i_* \hat{X})|_{\hat{u}}} \\
      =\ & \sup_{\+\sigma \in \bS^{d-1}}
      \sup_{X|_{i(\hat{u})} \in i_* T_{\hat{u}} \hat{M}} 
      \frac{g(\+\sigma \cdot \+A X, X)|_{i(\hat{u})}}
      {g(X, X)|_{i(\hat{u})}} \le \sup_{\+\sigma \in \bS^{d-1}}
      \sup_{X|_{i(\hat{u})} \in T_u M} 
      \frac{g(\+\sigma \cdot \+A X, X)|_{i(\hat{u})}}
      {g(X, X)|_{i(\hat{u})}} \\
      =\ & \rho(\+A|_{i(\hat{u})}) < +\infty,
    \end{aligned}
  \end{displaymath}
  which completes the proof.
\end{proof}

\subsubsection{Linear stability}
Let us investigate whether the natural model reduction preserves
the linear stability conditions. The first is to study whether 
\Cref{assump:equilibrium} can be preserved.
We need to make the following assumption, which is equivalent to  
preserving the equilibrium set and is easily verifiable given 
the reduced model.

\begin{assumption}
  \label{assump:equilibrium_reduced}
  The following assumptions hold:
  \begin{itemize}
    \item The set $\iota(\tilde{M})$ is a subset of $i(\hat{M})$;
    \item For each $\hat{u} \in \hat{M}$, the following relation 
    holds:
    \begin{equation}\label{eq:equilibriumND}
      p_* Q|_{i(\hat{u})} = 0 \in T_{\hat{u}} \hat{M} 
      \ \Longleftrightarrow\ 
      Q |_{i(\hat{u})} = 0 \in T_{i(\hat{u})} M.
    \end{equation}
  \end{itemize}
\end{assumption}

\begin{lemma}\label{thm:equilibrium}
  Under \Cref{assump:equilibrium_reduced},
  if the original equation \eqref{eq:1stonM} satisfies 
  \Cref{assump:equilibrium} on $(M, g)$,
  then the reduced model \eqref{eq:reduced} satisfies 
  \Cref{assump:equilibrium} on $(\hat{M}, \hat{g})$.
\end{lemma}

\begin{proof}
  \Cref{assump:equilibrium_reduced} yields that the equilibrium set 
  of \cref{eq:reduced} is also $\tilde{M}$. 
  Note that the embedding mapping $i$ yields a diffeomorphism 
  between $\hat{M}$ and $i(\hat{M}) \subset M$ and that 
  $\iota(\tilde{M}) \subset i(\hat{M})$. Therefore, one can define 
  the mapping $\hat{\iota}$ from $\tilde{M}$ to $\hat{M}$ as 
  \begin{displaymath}
    \hat{\iota} : \tilde{u} \in \tilde{M} \mapsto 
    i^{-1} ( \iota(\tilde{u}) ) \in \hat{M},
  \end{displaymath}
  which yields a diffeomorphism between $\tilde{M}$ and 
  $\hat{\iota}(\tilde{M})$. The following commutative diagrams hold: 
  $\iota = i \circ \hat{\iota}$ and 
  $\iota_*|_{\tilde{u}} = i_*|_{\hat{\iota}(\tilde{u})} 
  \circ \hat{\iota}_*|_{\tilde{u}}$ for each 
  $\tilde{u} \in \tilde{M}$. Since the pushforward mapping 
  $\iota_*|_{\tilde{u}}$ is injective, the pushforward mapping 
  $\hat{\iota}_*|_{\tilde{u}}$ is also injective. Therefore, 
  the equilibrium manifold $\tilde{M}$ is an embedded submanifold 
  of the ansatz manifold $\hat{M}$.

  As for the second item of \Cref{assump:equilibrium}, let us define 
  the projection $\hat{\pi} := \pi \circ i : \hat{M} \to \tilde{M}$, 
  which satisfies that 
  $\hat{\pi} \circ \hat{\iota} = \pi \circ i \circ \hat{\iota} = 
  \pi \circ \iota = \id_{\tilde{M}}$
  and that 
  \begin{multline*}
    \hat{g}(\hat{X}, \hat{\iota}_* \tilde{Y}) 
    |_{\hat{\iota}(\tilde{u})} = 
    g(i_* \hat{X}, i_* \hat{\iota}_* \tilde{Y}) 
    |_{i(\hat{\iota}(\tilde{u}))} \\
    = g(i_* \hat{X}, \iota_* \tilde{Y}) 
    |_{\iota(\tilde{u})} = 
    \tilde{g}(\pi_* i_* \hat{X}, \tilde{Y}) 
    |_{\tilde{u}} = 
    \tilde{g}(\hat{\pi}_* \hat{X}, \tilde{Y}) |_{\tilde{u}},
  \end{multline*}
  for each $\tilde{u} \in \tilde{M}$, $\hat{X} \in \Gamma(T \hat{M})$
  and $\tilde{Y} \in \Gamma(T \tilde{M})$. This observation 
  completes the proof of this lemma.
\end{proof}

\begin{theorem}
  Under \Cref{assump:equilibrium_reduced},
  if the original equation \eqref{eq:1stonM} satisfies 
  GSC (resp. GUSC or GWSC) on $(M, g)$, then 
  the reduced model \eqref{eq:reduced} satisfies 
  GSC (resp. GUSC or GWSC) on $(\hat{M}, \hat{g})$.
\end{theorem}

\begin{proof}
  The reduced model \eqref{eq:reduced} satisfies 
  \cref{eq:stability_1_1} in \Cref{item:stability_1} 
  thanks to \Cref{rmk:stability_1_1} and \Cref{thm:equilibrium}.

  For each $\hat{u} \in \hat{M}$ and 
  $\hat{Y} \in \Gamma(T \hat{M})$, one has that 
  \begin{displaymath}
    g(i_* \hat{Q}, i_* \hat{Y}) |_{i(\hat{u})} = 
    g(Q, i_* \hat{Y}) |_{i(\hat{u})}.
  \end{displaymath}
  Denote the Levi-Civita connection on $(\hat{M}, \hat{g})$ by 
  $\hat{D}$.
  Gauss--Codazzi formula yields that 
  \begin{displaymath}
    \hat{D}_{\hat{X}} \hat{Y} = p_* D_{i_* \hat{X}} (i_* \hat{Y}),
  \end{displaymath}
  for each $\hat{X}, \hat{Y} \in \Gamma(T \hat{M})$. 
  Therefore, for each $\tilde{u} \in \tilde{M}$ and 
  $\hat{X} \in \Gamma(T \hat{M})$, 
  \begin{multline*}
    \quad\ D_{i_* \hat{X}} 
      \left( g(i_* \hat{Q}, i_* \hat{Y}) \right) 
      \Big|_{\iota(\tilde{u})} = 
      g \left( D_{i_* \hat{X}} (i_* \hat{Q}), i_* \hat{Y} \right)
      \Big|_{\iota(\tilde{u})} = 
      \hat{g} (\hat{D}_{\hat{X}} \hat{Q}, \hat{Y}) 
      |_{\hat{\iota}(\tilde{u})} \\
    = D_{i_* \hat{X}} \left( g(Q, i_* \hat{Y}) \right) 
      \Big|_{\iota(\tilde{u})} = 
      g( D_{i_* \hat{X}} Q, i_* \hat{Y} ) 
      |_{\iota(\tilde{u})}.
  \end{multline*}
  For each $\tilde{Y} \in \Gamma(T \tilde{M})$, 
  let us take $\hat{Y} = \hat{\iota}_* \tilde{Y}$. 
  Since the original equation \eqref{eq:1stonM} satisfies 
  \cref{eq:stability_1_2}, one can obtain that 
  \begin{displaymath}
    \tilde{g} ( \hat{\pi}_* \hat{D}_{\hat{X}} \hat{Q}, \tilde{Y} )
    |_{\tilde{u}} = 
    \hat{g} ( \hat{D}_{\hat{X}} \hat{Q}, \hat{\iota}_* \tilde{Y} )
    |_{\hat{\iota}(\tilde{u})} = 
    g( D_{i_* \hat{X}} Q, {\iota}_* \tilde{Y} ) 
    |_{\iota(\tilde{u})} = 
    \tilde{g}( \pi_* D_{i_* \hat{X}} Q, \tilde{Y} ) 
    |_{\tilde{u}} = 0,
  \end{displaymath}
  which yields that 
  $\hat{\pi}_* \hat{D}_{\hat{X}} \hat{Q} |_{\tilde{u}} = 0$.
  Thus, the reduced model \eqref{eq:reduced} also satisfies 
  \cref{eq:stability_1_2}.

  As for \Cref{item:stability_2}, for each 
  $\hat{X}, \hat{Y} \in \Gamma(T \hat{M})$ and $j = 1, 2, \ldots, d$,
  one has that 
  \begin{displaymath}
    \hat{g} (\hat{A}^j \hat{X}, \hat{Y}) = 
    g( A^j i_* \hat{X}, i_* \hat{Y} ) = 
    g( i_* \hat{X}, A^j i_* \hat{Y} ) = 
    \hat{g} (\hat{X}, \hat{A}^j \hat{Y}).
  \end{displaymath}

  Let us now consider the preservation of Items 
  \ref{item:stability_3}, \ref{item:uniform_stability} or 
  \ref{item:weak_stability}.
  Note that the following equality holds: 
  $g(i_* \hat{Q}, i_* \hat{X}) |_{i(\hat{u})} = 
  g(Q, i_* \hat{X}) |_{i(\hat{u})}$ for each 
  $\hat{u} \in \hat{M}$ and $\hat{X} \in \Gamma(T \hat{M})$.
  Therefore, for each $\tilde{u} \in \tilde{M}$, one has that 
  \begin{multline*}
    D_{i_* \hat{X}} \left( g(i_* \hat{Q}, i_* \hat{X}) \right) 
      \Big|_{\iota(\tilde{u})} = 
      g \big( D_{i_* \hat{X}} (i_* \hat{Q}), i_* \hat{X} \big) 
      \big|_{\iota(\tilde{u})} = 
      \hat{g}(\hat{D}_{\hat{X}} \hat{Q}, \hat{X}) 
      |_{\hat{\iota}(\tilde{u})} \\
    = D_{i_* \hat{X}} \left( g(Q, i_* \hat{X}) \right) 
    \Big|_{\iota(\tilde{u})} = 
    g(D_{i_* \hat{X}} Q, i_* \hat{X}) |_{\iota(\tilde{u})}.
  \end{multline*}
  If \cref{eq:1stonM} satisfies \Cref{item:stability_3}, then 
  for each $\hat{X} \in \Gamma(T \hat{M})$ satisfying that 
  $\hat{X} |_{\hat{\iota}(\tilde{u})} \ne 0$ and 
  $\hat{\iota}_* \hat{\pi}_* \hat{X} |_{\hat{\iota}(\tilde{u})} = 0$, 
  one has that $i_* \hat{X} |_{\iota(\tilde{u})} \ne 0$, 
  \begin{displaymath}
    \iota_* \pi_* i_* \hat{X} |_{\iota(\tilde{u})} =
    \iota_* \hat{\pi}_* \hat{X} |_{\iota(\tilde{u})} = 0,
  \end{displaymath}
  and thus,
  \begin{displaymath}
    g(D_{i_* \hat{X}} Q, i_* \hat{X}) |_{\iota(\tilde{u})} < 0;
  \end{displaymath}
  If \cref{eq:1stonM} satisfies \Cref{item:uniform_stability}, 
  then 
  \begin{displaymath}
    \begin{aligned}
      & \quad\ g(D_{i_* \hat{X}} Q, i_* \hat{X}) 
      |_{\iota(\tilde{u})} \le - \lambda 
      \|{i_* \hat{X} - \iota_* \pi_* i_* \hat{X}}\|_g^2 
      \big|_{\iota(\tilde{u})} \\ 
      & = - \lambda \|{i_* \hat{X} - 
      i_* \hat{\iota}_* \hat{\pi}_* p_* i_* \hat{X}}\|
      _g^2 \big|_{\iota(\tilde{u})} = - \lambda 
      \|{\hat{X} - \hat{\iota}_* \hat{\pi}_* \hat{X}}\|
      _{\hat{g}}^2 \big|_{\hat{\iota}(\tilde{u})};
    \end{aligned}
  \end{displaymath}
  If \cref{eq:1stonM} satisfies \Cref{item:weak_stability}, then 
  \begin{displaymath}
    g(D_{i_* \hat{X}} Q, i_* \hat{X}) |_{\iota(\tilde{u})} \le 0.
  \end{displaymath}
  Therefore, the proof is completed.
\end{proof}

\subsection{Coordinate form}
\label{sec:reduction_coordinate}
In many cases, we are interested in not the explicit expression of 
the solution $u$ but rather some functions of $u$.
Suppose that \cref{eq:1stonM} is hyperbolic w.r.t. the Riemannian 
metric $g$, and we are interested in the value of 
smooth function $\+\omega = 
(\omega^1, \omega^2, \ldots, \omega^n)^\top : M \to \bR^n$.
Let us begin by seeing the evolution equation of $\+\omega$.
We need an assumption that the mapping $\+\omega$, when 
restricted to an open set $\hat{U} \subset \hat{M}$ we are 
considering, is a coordinate chart. 
At this time, for each $\hat{u} \in \hat{U}$, 
one basis of the tangent space $T_{\hat{u}} \hat{M}$ can be chosen 
as $\big\{ \tfrac{\partial \hat{u}}{\partial \omega^k} 
\big|_{\hat{u}} \big\}_{k = 1}^{n}$, where 
$\tfrac{\partial \hat{u}}{\partial \omega^k}$ is a tangent vector 
field on $\hat{U}$.
According to the expression of the reduced model \eqref{eq:reduced}, 
one can obtain that 
\begin{displaymath}
  \pd{\hat{u}}{\omega^{\ell}} \pd{\omega^{\ell}}{t} + 
    \hat{A}^j \pd{\hat{u}}{\omega^{\ell}}\pd{\omega^{\ell}}{x^j} 
    = \hat{Q} |_{\hat{u}}.
\end{displaymath}
Therefore, a closed system w.r.t. $\+\omega$ can be given by
the following symmetric hyperbolic equation,
\begin{displaymath}
  \begin{aligned}
    & \hat{g} \left( \pd{\hat{u}}{\omega^k}, 
      \pd{\hat{u}}{\omega^{\ell}} \right) \pd{\omega^{\ell}}{t} + 
      \hat{g} \left( \pd{\hat{u}}{\omega^k}, \hat{A}^j 
      \pd{\hat{u}}{\omega^{\ell}} \right)\pd{\omega^{\ell}}{x^j} = 
      \hat{g} \left( \pd{\hat{u}}{\omega^k}, 
      \hat{Q} |_{\hat{u}(\+\omega)} \right), \\
    & \omega^k(0, \cdot) = \omega^k \circ p \circ u_0,
  \end{aligned}
\end{displaymath}
which can be written in a more compact form, 
\begin{equation}\label{eq:reduce_compact}
  \begin{aligned}
    & \mathbf{A}^0 \pd{\+\omega}{t} + 
      \mathbf{A}^j \pd{\+\omega}{x^j} = 
      \mathbf{Q}, \\
    & \+\omega(0, \cdot) = \+\omega \circ p \circ u_0,
  \end{aligned}
\end{equation}
where $\mathbf{A}^0_{k, \ell} := \hat{g} 
\big( \tfrac{\partial \hat{u}}{\partial \omega^k}, 
\tfrac{\partial \hat{u}}{\partial \omega^{\ell}} \big)$, 
$\mathbf{A}^j_{k, \ell} := \hat{g} 
\big( \tfrac{\partial \hat{u}}{\partial \omega^k}, 
\hat{A}^j \tfrac{\partial \hat{u}}{\partial \omega^{\ell}} \big)$, 
and $\mathbf{Q}_k := \hat{g} 
\big( \tfrac{\partial \hat{u}}{\partial \omega^k}, 
\hat{Q} |_{\hat{u}(\+\omega)} \big)$.
Here, $\mathbf{A}^0$ is a symmetric and positive definite matrix, 
and $\mathbf{A}^j$ is a symmetric matrix for each 
$j = 1, 2, \ldots, n$.

\section{Application to kinetic equations}
\label{sec:kinetic}
In the previous section, we presented the abstract theory for model 
reduction of the first-order equations valued on manifolds. 
This section applies this theory to kinetic equations, a 
fundamental class of first-order equations valued on manifolds. We 
aim to demonstrate how the abstract theory can be used to develop 
model reduction methods for kinetic equations in a general form.

\subsection{A general form of kinetic equations}
\label{sec:kinetic_general}

Let us consider a general kinetic equation as an evolving law 
for a phase density function $f(\+\xi; t, \+x) \in \bR$ in time $t$,
\begin{equation}\label{eq:gke}
  \begin{aligned}
    & \pd{f}{t} + \+v(\+\xi) \cdot \nabla_{\+x}f = Q[f], \\
    & f(\cdot; 0, \cdot) = f_0,
  \end{aligned}
\end{equation}
where $t \in J \subset \bR$ is the time variable, 
$\+x = (x^1, x^2, \ldots, x^d) \in \Omega = \bR^d$ is the spatial 
variables, $\+\xi \in \+\Xi$ is the ordinate variables 
lying on a manifold $\+\Xi$ with dimension $\hat{d}$, and 
$\+v(\+\xi) = (v^1(\+\xi), v^2(\+\xi), \cdots, v^d(\+\xi))$ is the 
spatial transport velocity.
We adopt the conventional notation for kinetic equations and replace 
$u$ with $f$ to denote the solution in this context.
The phase density function $f$ depends on $t$, $\+x$, and $\+\xi$ 
with a 
total dimension of $1 + d + \hat{d}$. The right-hand term $Q[f]$ 
represents all other physical factors except spatial transport, 
which can introduce ordinate acceleration due to long-range effects 
such as Coulomb interaction or local particle interactions such as 
binary collisions.

This general kinetic equation encompasses many important examples, 
including the Boltzmann equation \cite{Boltzmann}, 
the radiative transfer equation \cite{chandrasekhar}, 
the Landau equation \cite{Landau1965}, 
the Vlasov equation \cite{Vlasov}, 
the Wigner equation \cite{Wigner}, 
the kinetic Fokker-Planck equation \cite{Chandrasekhar1943}, 
the Peierls--Boltzmann equation \cite{peierls}, and the uncertainty 
quantification of transport equations \cite{JinXiuZhu2015}.

The formulation presented here can be connected to the abstract 
theory introduced in the previous section. The manifold 
$M$ comprises all admissible solutions $f(\cdot; t, \+x)$, 
forming an open set in a topological vector space $V$. 
In our case, the linear space $V$ is a function space of real-valued 
functions defined on the manifold $\+\Xi$.
The Riemannian metric $g$ should satisfy specific requirements, as 
stated in \Cref{thm:bilinear_hyperbolicity} and \Cref{assump:metric}
later, to ensure hyperbolicity and linear stability. 
The tangent vectors of $M$ lie in the space $V$, and the 
$(1,1)$-tensor $A^j$ acts on a tangent vector in $V$ as 
multiplication by the function $v^j(\+\xi)$. 
The right-hand term $Q[f]$ is a tangent vector at $f \in M$. We 
always suppose that the solutions to \cref{eq:gke} 
decay sufficiently fast at infinity.

\subsection{Structural properties}
\label{sec:kinetic_properties}
We demonstrated that the model reduction preserves several important 
structural properties in \Cref{sec:reduction_preserve}. To prove 
these properties for the reduced models, it suffices to show that 
the kinetic equations respect these structures. 
Let us expand the framework proposed in 
\Cref{sec:framework_structure} 
to incorporate these structural properties of kinetic equations.

\subsubsection{Hyperbolicity}
\label{sec:kinetic_hyperbolicity}
We begin our analysis by examining the property of hyperbolicity 
of kinetic equations. The conditions the Riemannian metric 
should satisfy can be identified based on the structure of 
\cref{eq:gke}. To facilitate further discussions, we 
introduce the following lemma, the proof of which is provided in 
\Cref{app:kinetic_properties}.

\begin{lemma}\label{thm:bilinear}
  Let $\+v \in C^\infty(\+\Xi, \bR^d)$ be the smooth injective 
  immersion from the finite-dimensional manifold $\+\Xi$ to $\bR^d$,
  and $V = C^\infty_{\rc}(\+\Xi)$ be the space of smooth functions 
  defined on $\+\Xi$ with compact support.
  If $a : V \times V \to \bR$ is a continuous 
  bilinear form on $V$, then the following statements are equivalent:
  \begin{enumerate}[label=(\roman*),ref=(\roman*)]
    \item \label{item:bilinear_1}
    For each $h_1, h_2 \in V$, and 
    $j = 1, 2, \ldots, d$, the following equality holds:
    \begin{displaymath}
      a(v^j h_1, h_2) = 
      a(h_1, v^j h_2);
    \end{displaymath}
    \item \label{item:bilinear_2}
    There exists $\cA \in V^*$ such that 
    for each $h_1, h_2 \in V$, the following equality holds:
    \begin{displaymath}
      a(h_1, h_2) = \inn{\cA, h_1 h_2}.
    \end{displaymath}
  \end{enumerate}
\end{lemma}

From \Cref{thm:bilinear}, the following theorem can be derived 
directly, which provides a necessary and sufficient 
condition for the hyperbolicity of the kinetic equation 
\eqref{eq:gke} w.r.t. the Riemannian metric $g$.

\begin{theorem}\label{thm:bilinear_hyperbolicity}
  Under the assumptions of \Cref{thm:bilinear}, 
  the following statements are equivalent for 
  $g \in \Gamma(T^{0, 2} M)$:
  \begin{enumerate}[label=(\roman*),ref=(\roman*)]
    \item \label{item:metric_1}
    The kinetic equation \eqref{eq:gke} is hyperbolic w.r.t. the 
    Riemannian metric $g$;
    \item \label{item:metric_2}
    There exists a mapping $\cA : M \to V^*$ such that $\cA(f)$ is 
    strictly positive for each $f \in M$ and that 
    \begin{displaymath}
      g(h_1, h_2) = \inn{\cA(f), h_1 h_2},
    \end{displaymath}
    for each $h_1, h_2 \in T_{f} M$;
    \item \label{item:metric_3}
    There exists a strictly positive Radon measure $\mu^f$ on 
    the manifold $\+\Xi$ for each $f \in M$ such that 
    \begin{equation}\label{eq:kinetic_hyperbolicity}
      g(h_1, h_2) = \int_{\+\Xi} h_1 h_2 \,\rd \mu^f(\+\xi), 
    \end{equation}
    for each $h_1, h_2 \in T_f M$.
  \end{enumerate}
\end{theorem}

\begin{proof}
  The equivalence between \ref{item:metric_1} and \ref{item:metric_2}
  is straightforward by \Cref{thm:bilinear}. It suffices to show that 
  \ref{item:metric_2} yields \ref{item:metric_3}. 
  Recall that a functional $\cA \in V^*$ is strictly 
  positive if the inequality $\inn{\cA, h} \ge 0$ holds
  for each $h \in V$ satisfying $h \ge 0$,
  and the equality sign holds if and 
  only if $h = 0$. At this time, for each compact set 
  $K \subset \+\Xi$, one can choose $\psi_K \in V$ such that 
  $\psi_K \in [0, 1]$ and $\psi_K|_{K} \equiv 1$. Define 
  \begin{displaymath}
    B_K := \left\{ h \in V \,\Big|\, \mathrm{supp}(h) \subset K, \ 
    \norm{h}_{C^0(\+\Xi)} \le 1 \right\}.
  \end{displaymath}
  Since $\cA \in V^*$ is strictly positive, one has that 
  \begin{displaymath}
    \sup_{h \in B_K} \inn{\cA, h} \le \inn{\cA, \psi_K} < +\infty.
  \end{displaymath}
  Therefore, the functional $\cA \in V^*$ can be extended to 
  a strictly positive linear functional on $C_c(\+\Xi)$. 
  By Riesz--Markov--Kakutani representation theorem, there exists a 
  unique strictly positive Radon measure $\mu$ on $\+\Xi$ such that 
  $\inn{\cA, h} = \int_{\+\Xi} h(\+\xi) \,\rd \mu(\+\xi)$ for each 
  $h \in C_c(\+\Xi)$. Therefore, there exists a unique strictly 
  positive Radon measure $\mu^f$ on $\+\Xi$ for each $f \in M$ such 
  that $\inn{\cA(f), h} = \int_{\+\Xi} h(\+\xi) \,\rd \mu^f(\+\xi)$ 
  for each $h \in C_c(\+\Xi)$, which completes the proof of 
  \Cref{thm:bilinear_hyperbolicity}.
\end{proof}

\begin{remark}
  The function space $V$ in \Cref{thm:bilinear} is supposed to be 
  $V = C_c^\infty(\+\Xi)$. In practical applications, 
  the space $V$ should be a broader function space encompassing 
  $C_c^\infty(\+\Xi)$ as a subspace. 
  At this time, the conclusion of \Cref{thm:bilinear_hyperbolicity} 
  still holds.
  Suppose that the kinetic equation \eqref{eq:gke} is hyperbolic 
  w.r.t. the Riemannian metric $g$.
  In that case, the completion of the 
  tangent space $T_f M$ w.r.t. the inner product $g|_f$ is 
  $L^2(\+\Xi, \mu^f)$.
  This result 
  means that any information about the derivatives w.r.t. the 
  ordinate variables $\+\xi$ does not have to be considered 
  when choosing the Riemannian metric for model reduction.
\end{remark}

\subsubsection{Conservation laws and entropy dissipation}
When studying kinetic equations, it is often crucial to investigate 
the existence of conserved quantities or entropy. A necessary 
condition for a quantity to be a conserved quantity or entropy is 
that it has a flux. In this part, we employ techniques discussed 
in \Cref{sec:kinetic_hyperbolicity} to characterize quantities with 
fluxes, summarized by the following theorem.

\begin{theorem}\label{thm:bilinear_flux}
  Under the assumptions of \Cref{thm:bilinear}, if the manifold $M$ 
  is a simply connected open subset of the space $V$, then 
  the following statements are equivalent:
  \begin{enumerate}[label=(\roman*),ref=(\roman*)]
    \item \label{item:flux_1}
    The function $c \in C^\infty(M)$ possesses a flux;
    \item \label{item:flux_2}
    There exists a mapping $\cA : M \to V^*$ such that 
    for each $f \in M$ and $h_1, h_2 \in V$, one has that
    \begin{equation}\label{eq:bilinear_flux}
      \partial^2 c(f; h_1, h_2) = \inn{\cA(f), h_1 h_2},
    \end{equation}
    where $\partial^2 c$ is the second-order Gateaux derivative 
    of $c$.
  \end{enumerate}
\end{theorem}

\begin{proof}
  Let us prove \Cref{thm:bilinear_flux} in the following order:
  \ref{item:flux_1} $\Rightarrow$ \ref{item:flux_2} $\Rightarrow$ 
  \ref{item:flux_1}.

  \ref{item:flux_1} $\Rightarrow$ \ref{item:flux_2}. 
  By definition, there exists a function 
  $\+F_c \in C^\infty(M, \bR^d)$ such that for each $f \in M$ and 
  $h \in V$, one has that 
  $\partial F_c^j(f; h) = \partial c(f; v^j h)$, which yields that 
  \begin{displaymath}
    \partial^2 c(f; v^j h_1, h_2) = \partial^2 F_c^j(f; h_1, h_2) = 
    \partial^2 F_c^j(f; h_2, h_1) = \partial^2 c(f; v^j h_2, h_1).
  \end{displaymath}
  By \Cref{thm:bilinear}, the statement \ref{item:flux_2} holds.

  \ref{item:flux_2} $\Rightarrow$ \ref{item:flux_1}. 
  Let us choose a function $f_0 \in M$. For each $f \in M$, there 
  exists a smooth path $\gamma : [0, 1] \to M$ such that 
  $\gamma(0) = f_0$ and $\gamma(1) = f$. Define 
  \begin{displaymath}
    F^j_{c}(f) = \int_{0}^{1} 
    \partial c(\gamma(\tau); v^j \gamma'(\tau)) \,\rd \tau.
  \end{displaymath}
  It suffices to show that this definition is independent of the 
  choice of the path. Suppose that two paths, $\gamma_0$ and 
  $\gamma_1$, satisfy the conditions above. Since the manifold $M$ 
  is simply connected, there exists a smooth homotopy 
  $\gamma : [0, 1] \times [0, 1] \to M$ such that 
  $\gamma(0, \cdot) = \gamma_0$, $\gamma(1, \cdot) = \gamma_1$, 
  $\gamma(\cdot, 0) \equiv f_0$, and $\gamma(\cdot, 1) \equiv f$.
  Therefore, one obtains that 
  \begin{displaymath}
    \begin{aligned}
      & \pd{}{s} \int_{0}^{1} \partial c \big( \gamma(s, \tau); v^j 
      \tfrac{\partial \gamma}{\partial \tau} (s, \tau) \big) 
      \,\rd \tau \\
      =\ & \int_{0}^{1} \partial^2 c \big( \gamma(s, \tau); v^j 
      \tfrac{\partial \gamma}{\partial \tau} (s, \tau), 
      \tfrac{\partial \gamma}{\partial s} (s, \tau) \big) + 
      \partial c \big( \gamma(s, \tau); v^j 
      \tfrac{\partial^2 \gamma}{\partial \tau \partial s} (s, \tau) 
      \big) \,\rd \tau \\
      =\ & \int_{0}^{1} \pd{}{\tau} 
      \partial c \big( \gamma(s, \tau); v^j 
      \tfrac{\partial \gamma}{\partial s} (s, \tau) \big) 
      \,\rd \tau = \partial c \big( f; v^j 
      \tfrac{\partial \gamma}{\partial s} (s, 1) \big) - 
      \partial c \big( f_0; v^j 
      \tfrac{\partial \gamma}{\partial s} (s, 0) \big) \\ 
      =\ & 0,
    \end{aligned}
  \end{displaymath}
  which completes the proof of this theorem.
\end{proof}

\begin{remark}
  Roughly speaking, the meaning of \cref{eq:bilinear_flux} is that 
  the Hessian operator $\partial^2 c$ of the quantity $c(f)$ 
  in the ordinate representation is analogous to a diagonal matrix.
  Suppose that the manifold $M$ is the Euclidean space $V = \bR^n$ 
  and that the Hessian matrix of the function $c : M \to \bR$ 
  satisfies \cref{eq:bilinear_flux}, where $h_1 h_2$ represents
  component-wise multiplication. 
  One can obtain that the Hessian matrix 
  $\partial^2 c$ is diagonal, and the function $c$ has the form of 
  $c(f) = c_k(f^k)$, where $f = (f^1, f^2, \ldots, f^n) \in M$, and 
  $c_k : \bR \to \bR$ is an arbitrary smooth function for each 
  $k = 1, 2, \ldots, n$.
\end{remark}

\begin{remark}
  According to \Cref{thm:bilinear_flux} and by induction, 
  if the quantity $c$ has a flux, then the following equality holds:
  \begin{equation}\label{eq:bilinear_flux_k}
    \partial^k c(f; h_1, h_2, \ldots, h_k) = 
    \inn{\cA_k(f), h_1 h_2 \cdots h_k}, 
  \end{equation}
  for each $k \ge 2$, where $\cA_2(f) = \cA(f)$ and 
  \begin{displaymath}
    \partial \cA_k(f; h) = \cA_{k+1}(f) \circ 
    [\text{multiplication operator by } h].
  \end{displaymath}
\end{remark}

The significance of \Cref{thm:bilinear_flux} is that it gives a 
verifiable way to determine whether a quantity has a flux. Let us 
consider the following examples.

\begin{example}
  The quantity $c$ in the form of 
  \begin{displaymath}
    c(f) = \inn{\cB, f^{\otimes k}},
  \end{displaymath}
  can never be a candidate for a conserved quantity or entropy, where 
  $k \ge 2$, the function $f^{\otimes k}$ is defined as 
  \begin{displaymath}
    f^{\otimes k} : (\+\xi_1, \+\xi_2, \ldots, \+\xi_k) \in \+\Xi^k 
    \mapsto f(\+\xi_1) f(\+\xi_2) \cdots f(\+\xi_k) \in \bR,
  \end{displaymath}
  the functional $\cB \in C^\infty_{\rc}(\+\Xi^k)^*$ is symmetric 
  under all permutations of its variables 
  $\+\xi_1, \+\xi_2, \ldots, \+\xi_k \in \+\Xi$,
  and the support set of $\cB$ is not included in the set 
  \begin{displaymath}
    \left\{ (\+\xi_1, \+\xi_2, \ldots, \+\xi_k) \in \+\Xi^k 
    \,\big|\, \+\xi_1 = \+\xi_2 = \cdots = \+\xi_k \in \+\Xi \right\}.
  \end{displaymath}
  At this time, the $k$th-order derivative of $c(f)$ is 
  \begin{multline*}
    \partial^k c(f; h_1, h_2, \ldots, h_k) \\
    = \sum_{\sigma \in S_k} 
    \inn{\cB, h_{\sigma(1)} \otimes h_{\sigma(2)} \otimes 
    \cdots \otimes h_{\sigma(k)}} 
    = k! \inn{\cB, h_1 \otimes h_2 \otimes \cdots \otimes h_k},
  \end{multline*}
  where $S_k$ is the symmetric group of degree $k$.
  By the assumption on the support set of $\cB$, one can choose 
  the functions $h_1, h_2, \ldots, h_k \in V$ such that the 
  intersection of all the support sets of $h_1, h_2, \ldots, h_k$ is 
  empty and that $\partial^k c(f; h_1, h_2, \ldots, h_k) \ne 0$, 
  which contradicts \cref{eq:bilinear_flux_k}. 
\end{example}

\begin{example}
  \label{exp:flux_convolution}
  Another example of a quantity without a flux is 
  \begin{displaymath}
    c(f) = \int_{\+\Xi} G(f * \phi) \,\rd \+\xi,
  \end{displaymath}
  where the manifold $\+\Xi$ is the Euclidean space $\bR^d$, the 
  smooth function $G$ is nonlinear such that $G''(f * \phi)$ does not 
  identically equal to zero, and the support set of 
  $\phi \in C^\infty_{\rc}(\bR^d)^*$ contains more than two points.
  At this time, the second-order derivative of $c(f)$ is 
  \begin{displaymath}
    \partial^2 c(f; h_1, h_2) = 
    \int_{\+\Xi} G''(f * \phi) (h_1 * \phi) (h_2 * \phi) \,\rd \+\xi.
  \end{displaymath}
  By the assumption on the function $G$ and the support set of $\phi$,
  one can choose the functions $h_1, h_2 \in V$ such that the 
  intersection of the support sets of $h_1$ and $h_2$ is empty and 
  that $\partial^2 c(f; h_1, h_2) \ne 0$, which contradicts 
  \cref{eq:bilinear_flux}.
\end{example}

Let us explore two fundamental categories of quantities 
satisfying \cref{eq:bilinear_flux}.

The first category of quantities with fluxes is given by 
\begin{equation}\label{eq:kinetic_conserve}
  c(f) = \inn{\fc, f}, 
\end{equation}
where $\fc \in V^*$ is a linear functional on the space $V$. 
The first-order derivative of $c(f)$ is
$\partial c(f; h) = \inn{\fc, h}$, and the second-order derivative
is given by \cref{eq:bilinear_flux} with $\cA(f) = 0$.
Therefore, the quantity in this representation has a flux, and the 
corresponding flux is $F^j(f) = \inn{\fc, v^j f}$. 
Such quantities in this category collectively form a linear space, 
and they are particularly suitable for the conserved quantities, 
which are one of the significant concerns in kinetic equations. 
\begin{definition}[Conserved quantities for kinetic equations]
  We say that the quantity $c$ in the form of 
  \cref{eq:kinetic_conserve} is a \textit{conserved quantity} for the 
  kinetic equation \eqref{eq:gke} if the equality  
  $\inn{\fc, Q[f]} = 0$ holds for each $f \in M$. 
  Moreover, we define the space $\bE$ of 
  \textit{collision invariants} for \cref{eq:gke} as 
  \begin{equation}\label{eq:collisionInv}
    \bE := \left\{ \fc \in V^* \,\big|\, \inn{\fc, Q[f]} = 0, 
    \forall f \in M \right\}.
  \end{equation}
\end{definition}

Another general category of quantities with fluxes 
can be characterized by the following integral representation,
\begin{equation}\label{eq:kinetic_entropy}
  c(f) = \int_{\+\Xi} \eta(\+\xi, f(\+\xi)) \,\rd \mu(\+\xi),
\end{equation}
where the measure $\mu$ on $\+\Xi$ and the mapping 
$\eta : \+\Xi \times \bR \to \bR$ are chosen to ensure that $c(f)$ is 
well-defined and that the mapping $\eta$ is smooth w.r.t. 
its second variable. 
The first-order derivative of $c(f)$ is $\partial c(f; h) = 
\int_\+\Xi \tfrac{\partial \eta}{\partial f}(\+\xi, f(\+\xi)) 
h(\+\xi) \,\rd\mu(\+\xi)$,
and the second-order derivative is 
\begin{displaymath}
  \partial^2 c(f; h_1, h_2) = \int_{\+\Xi} 
  \dfrac{\partial^2 \eta}{\partial f^2}(\+\xi, f(\+\xi))
  h_1(\+\xi) h_2(\+\xi) \,\rd \mu(\xi),
\end{displaymath}
following \cref{{eq:bilinear_flux}}, where the functional 
$\cA(f)$ is a multiplication operator by the function 
$\tfrac{\partial^2 \eta}{\partial f^2}(\+\xi, f(\+\xi))$.
The quantity $c(f)$ has a flux according to \Cref{thm:bilinear_flux}.
The corresponding flux is given by 
$F^j(f) = 
\int_{\+\Xi} v^j(\+\xi) \eta(\+\xi, f(\+\xi)) \,\rd\mu(\+\xi)$.
Furthermore, we assert that all quantities in this representation 
collectively form a linear space. Let us consider two such 
quantities, 
$c_1(f) = \int_{\+\Xi} \eta_1(\+\xi, f(\+\xi)) \,\rd \mu_1(\+\xi)$ 
and 
$c_2(f) = \int_{\+\Xi} \eta_2(\+\xi, f(\+\xi)) \,\rd \mu_2(\+\xi)$. 
The measures $\mu_1$ and $\mu_2$ are absolutely continuous w.r.t. the 
measure $\mu_1 + \mu_2$, which yields that there 
exist two functions $m_1, m_2 \in L^1(\+\Xi, \mu_1 + \mu_2)$ such 
that $\rd \mu_k (\+\xi) = m_k(\+\xi) \rd (\mu_1 + \mu_2)(\+\xi)$, 
where $k = 1, 2$. Thus, 
\begin{displaymath}
  (c_1 + c_2)(f) = \int_{\+\Xi} 
  \big( \eta_1(\+\xi, f(\+\xi)) m_1(\+\xi) + 
  \eta_2(\+\xi, f(\+\xi)) m_2(\+\xi) \big) 
  \,\rd (\mu_1 + \mu_2)(\+\xi),
\end{displaymath}
which yields that the quantities in this representation are closed 
under addition. Moreover, they are closed under scalar 
multiplication, which confirms the assertion.
This category holds particular relevance in 
the study of entropy. The H-theorem, considered essential due to 
Boltzmann's definition of entropy for gas-particle systems, plays a 
significant role in kinetic equations. For the kinetic equation 
\eqref{eq:gke} in the general form, let us define the entropy 
as follows.
\begin{definition}[Entropy for kinetic equations]
  \label{def:kinetic_entropy}
  We say that the quantity $c$ in the form of 
  \cref{eq:kinetic_entropy} is an 
  \textit{entropy} for the kinetic equation \eqref{eq:gke}
  if the inequality 
  $\int_\+\Xi \tfrac{\partial \eta}{\partial f}(\+\xi, f(\+\xi)) 
  Q[f] \,\rd\mu(\+\xi) \le 0$ holds
  for each $f \in M$. 
\end{definition}

These definitions are consistent with the conserved 
quantities and entropy defined in \Cref{sec:framework_flux}
for the quantities in the specific form of 
\cref{eq:kinetic_conserve,eq:kinetic_entropy}.

\begin{example}
  Let us give an example of the quantity $c(f)$ that satisfies 
  $\inn{\rd c, Q} \le 0$ but does not have a flux and thus is not 
  an entropy.
  Let us consider the Landau equation as follows,
  \begin{displaymath}
    \pd{f}{t} + \+\xi \cdot \nabla_{\+x} f = Q[f] := 
    \nabla_{\+\xi} \cdot \int_{\bR^d} A(\+\xi - \+\xi_*) 
    \left( \nabla_{\+\xi} \frac{\delta H}{\delta f} - 
    \nabla_{\+\xi_*} \frac{\delta H_*}{\delta f_*} \right) f f_*
    \,\rd \+\xi_*,
  \end{displaymath}
  where $\+\xi \in \+\Xi = \bR^d$ is the ordinate variables,
  $f = f(\+\xi)$ and $f_* = f(\+\xi_*)$ are used,  
  the matrix $A(\+z) := |\+z|^\gamma (|\+z^2| I_d - \+z \otimes \+z)$
  with $-d-1 \le \gamma \le 1$
  is symmetric and positive semi-definite, and 
  $H(f) = \int_{\bR^d} f \log f \,\rd \+\xi$ is the entropy since 
  it takes the form of \cref{eq:kinetic_entropy} and satisfies 
  \begin{multline*}
    \inn{\partial H(f), Q[f]} = - \int_{\bR^d} \int_{\bR^d} 
    \left( \nabla_{\+\xi} \frac{\delta H}{\delta f} - 
    \nabla_{\+\xi_*} \frac{\delta H_*}{\delta f_*} \right) \cdot \\
    A(\+\xi - \+\xi_*) 
    \left( \nabla_{\+\xi} \frac{\delta H}{\delta f} - 
    \nabla_{\+\xi_*} \frac{\delta H_*}{\delta f_*} \right) f f_*
    \,\rd \+\xi_* \,\rd \+\xi \le 0.
  \end{multline*}
  In \cite{Carrillo2020particle}, they considered the regularized 
  spatially homogeneous Landau equation as follows to discuss their 
  new particle method, 
  \begin{displaymath}
    \pd{f}{t} = Q_{\epsilon}[f] := 
    \nabla_{\+\xi} \cdot \int_{\bR^d} A(\+\xi - \+\xi_*) 
    \left( \nabla_{\+\xi} \frac{\delta H_{\epsilon}}{\delta f} - 
    \nabla_{\+\xi_*} \frac{\delta H_{\epsilon,*}}{\delta f_*} 
    \right) f f_* \,\rd \+\xi_*,
  \end{displaymath}
  where the regularized version of $H(f)$ is 
  \begin{displaymath}
    H_{\epsilon}(f) = \int_{\bR^d} (f * \psi_\epsilon) 
    \log(f * \psi_\epsilon) \,\rd \+\xi, 
  \end{displaymath}
  and the mollifier $\psi_\epsilon$ is 
  \begin{displaymath}
    \psi_\epsilon(\+\xi) = \frac{1}{(2 \pi \epsilon)^{d / 2}}
    \exp \left( - \frac{|\+\xi|^2}{2 \epsilon} \right).
  \end{displaymath}
  This strategy was also used in the nonlinear Fokker--Planck 
  equation \cite{Carrillo2019blob}. 
  At this time, one can obtain that 
  \begin{multline*}
    \inn{\partial H_\epsilon(f), Q_\epsilon[f]} 
    = - \int_{\bR^d} \int_{\bR^d} 
    \left( \nabla_{\+\xi} \frac{\delta H_\epsilon}{\delta f} - 
    \nabla_{\+\xi_*} \frac{\delta H_{\epsilon,*}}{\delta f_*} \right) 
    \cdot \\ A(\+\xi - \+\xi_*) 
    \left( \nabla_{\+\xi} \frac{\delta H_\epsilon}{\delta f} - 
    \nabla_{\+\xi_*} \frac{\delta H_{\epsilon,*}}{\delta f_*} \right)
    f f_* \,\rd \+\xi_* \,\rd \+\xi \le 0.
  \end{multline*}  
  However, the quantity $H_\epsilon(f)$ is not an entropy 
  for the regularized spatially inhomogeneous Landau equation since 
  it does not have a flux,
  according to \Cref{exp:flux_convolution}.
\end{example}

\subsubsection{Finite propagation speed}
For the sake of simplicity, let us suppose that the tangent space 
$T_f M$ is a Hilbert space w.r.t. the inner product $g|_f$.
By \Cref{thm:bilinear_hyperbolicity}, one has that 
$T_f M = L^2(\+\Xi, \mu^f)$. For each $\+\sigma \in \bS^{d-1}$, the 
operator $\+\sigma \cdot \+A|_f$ is the multiplication operator by 
the function $\+\sigma \cdot \+v$. Its spectrum is the 
essential range of the function $\+\sigma \cdot \+v$, i.e., 
\begin{displaymath}
  \left\{ \lambda \in \bR \,\Big|\, \mu^f \left( \left\{ \+\xi \in 
  \+\Xi \,\big|\, |\+\sigma \cdot \+v(\+\xi) - \lambda|  < 
  \varepsilon \right\} \right) > 0, \ \forall \varepsilon > 0  
  \right\}.
\end{displaymath}
Note that the mapping $\+\sigma \cdot \+v : \+\Xi \to \bR$ is 
continuous and that the measure $\mu^f$ is strictly positive. 
Therefore, the spectrum of $\+\sigma \cdot \+A |_f$ is exactly the 
range of the function $\+\sigma \cdot \+v$. We have proven the 
following proposition.

\begin{proposition}\label{thm:kinetic_speed}
  If the kinetic equation \eqref{eq:gke} is hyperbolic w.r.t. the 
  Riemannian metric $g$ and the tangent space $T_f M$ is a 
  Hilbert space w.r.t. the inner product $g|_f$ for each 
  $f \in M$, then the maximum propagation speed of \cref{eq:gke} is 
  \begin{displaymath}
    \rho(\+A|_f) = \sup_{\+\xi \in \+\Xi} |\+v(\+\xi)|.
  \end{displaymath}
\end{proposition}

\subsubsection{Linear stability and H-theorem}
\label{sec:kinetic_linear}
One of the main features of many collisional kinetic equations is 
their tendency to converge an equilibrium distribution,  
often characterized by H-theorem \cite{Villanireview}. 
We would like to consider the relation between the linear stability 
conditions and the H-theorem for the kinetic equation \eqref{eq:gke}.

\begin{definition}
  \label{def:Hthm}
  We say that the kinetic equation \eqref{eq:gke} satisfies 
  \textit{H-theorem} if there exists a real-valued function $H$ 
  defined on a convex open subset $U \subset V$ 
  containing the equilibrium 
  manifold $\tilde{M}$
  with strictly positive definite second-order Fr\'echet derivative 
  $\partial^2 H(f)$ for each $f \in U$
  such that the following conditions hold:
  \begin{enumerate}[label=H\arabic*.,ref=H\arabic*]
    \item \label{item:H1}
    For each $\tilde{f} \in \tilde{M}$, one has that 
    $\partial H(\tilde{f}) \in \bE$, where $\bE$ is the space of 
    collision invariants for \cref{eq:gke} 
    defined by \cref{eq:collisionInv}.

    \item \label{item:H2}
    For each $f \in U \cap M$, the following inequality holds,
    \begin{displaymath}
      S(f) := \inn{\partial H(f), Q[f]} \le 0,
    \end{displaymath}
    and the equality sign holds if and only if $f \in \tilde{M}$.
  \end{enumerate}
\end{definition}

\begin{remark}
  \label{rmk:Hthm}
  H-theorem says that the equilibrium distribution achieves the 
  minimum of the entropy under constraints imposed by the 
  conservation laws \cite{Villanireview}.
  The condition \ref{item:H2} yields that the quantity $H$ is an 
  entropy for \cref{eq:gke} according to \Cref{def:cqe}. 
  Suppose that a basis of the space $\bE$ is 
  $\{ \fc_i \,|\, i \in \cI \}$, where $\cI$ is the 
  index set. Let us consider the following optimization problem:
  \begin{equation}\label{eq:H-opt}
    \begin{aligned}
      \min_{f \in U} \ \ & H(f), \\ \text{s.t.} \ \ & 
      \inn{\fc_i, f} = c_i \in \bR, \quad \forall i \in \cI.
    \end{aligned}
  \end{equation}
  The critical points of this optimization problem are contained in 
  the equilibrium set,
  \begin{equation}\label{eq:H-equi}
    \tilde{M} = \{ f \in M \,|\, \partial H(f) \in \bE \},
  \end{equation}
  according to \Cref{def:Hthm}.
\end{remark}

We need to make the following assumption on the Riemannian metric $g$.

\begin{assumption}
  \label{assump:metric}
  The Riemannian metric $g$ satisfies \cref{eq:kinetic_hyperbolicity} 
  to ensure the hyperbolicity of the kinetic equation \eqref{eq:gke}, 
  and satisfies 
  \begin{equation}\label{eq:kinetic_natural}
    g(\cdot, \cdot) |_{\iota(\tilde{f})} = \partial^2 H(\tilde{f}),
  \end{equation}
  for each $\tilde{f} \in \tilde{M}$. 
\end{assumption}

\begin{remark}
  \Cref{eq:kinetic_natural} in \Cref{assump:metric} is natural when 
  the system has an entropy. For example, let us consider
  the case when the manifold $M$ is an open set in a 
  finite-dimensional Euclidean space. If the system has an entropy 
  $H(u)$, then the symmetrizer $A^0|_u$,
  which determines the Riemannian metric to ensure hyperbolicity, 
  should be the Hessian of the entropy, i.e., $\partial^2 H(u)$.
  This observation is classical due to 
  Godunov \cite{Godunov1961interesting},
  Friedrichs and Lax \cite{Friedrichs1971conservation}, and 
  Boillat \cite{Boillat1974} for conservation laws.
\end{remark}

Before investigating the linear stability conditions for the 
kinetic equation \eqref{eq:gke}, let us verify whether \cref{eq:gke}
satisfies \Cref{assump:equilibrium}, which H-theorem and 
\Cref{assump:metric} can guarantee.

\begin{lemma}
  \label{thm:HassumpEqui}
  Suppose that the Riemannian metric $g$ satisfies 
  \Cref{assump:metric}. If the kinetic equation \eqref{eq:gke} 
  satisfies H-theorem, then 
  \cref{eq:gke} satisfies \Cref{assump:equilibrium}.
\end{lemma}

\begin{proof}
  According to H-theorem, the equilibrium manifold $\tilde{M}$ is 
  given by \cref{eq:H-equi}.
  Given a distribution function $f \in M$, let us solve the convex 
  optimization problem \eqref{eq:H-opt} with the parameters
  $c_i = \inn{\fc_i, f}$ and obtain the optimal solution 
  $\pi(f) \in \tilde{M}$, which defines a projection 
  $\pi : M \to \tilde{M}$ satisfying that 
  $\pi \circ \iota = \id_{\tilde{M}}$.
  For each $f \in U \cap M$ and each $h \in T_f M$, choose a curve 
  $\kappa \mapsto f^\kappa = f + \kappa h$ in $U \cap M$. 
  By the definition of the projection $\pi$, one can obtain that 
  $\inn{\fc, \iota(\pi(f^\kappa))} = \inn{\fc, f^\kappa}$
  and thus, $\inn{\fc, \iota_* \pi_* h} = \inn{\fc, h}$.
  For each $\tilde{f} \in \tilde{M}$ and each 
  $\tilde{h} \in T_{\tilde{f}} \tilde{M}$, choose a curve 
  $\kappa \mapsto \tilde{f}^\kappa$ in $\tilde{M}$ satisfying that 
  \begin{displaymath}
    \tilde{f}^0 = \tilde{f}, \quad 
    \od{\tilde{f}^\kappa}{\kappa} \bigg|_{\kappa = 0} = \tilde{h}.
  \end{displaymath}
  By \cref{eq:H-equi}, one has that 
  $\partial H(\tilde{f}^\kappa) \in \bE$, which yields that 
  \begin{equation}\label{eq:H-bE}
    g(\iota_* \tilde{h}, \cdot)|_{\iota(\tilde{f})} = 
    \partial^2 H(\tilde{f})(\tilde{h}, \cdot) \in \bE,
  \end{equation}
  thanks to \Cref{assump:metric}. Therefore, 
  \begin{displaymath}
    \tilde{g}(\tilde{h}, \pi_* h) |_{\tilde{f}} = 
    g(\iota_* \tilde{h}, h) |_{\iota(\tilde{f})},
  \end{displaymath}
  for each $\tilde{f} \in \tilde{M}$, $h \in T_{\iota(\tilde{f})} M$,
  and $\tilde{h} \in T_{\tilde{f}} \tilde{M}$, which completes the 
  proof.
\end{proof}

The following theorem, roughly speaking, 
says that the H-theorem yields GWSC.

\begin{theorem}\label{thm:H->GWSC}
  Suppose that the Riemannian metric $g$ satisfies 
  \Cref{assump:metric}. If the kinetic equation \eqref{eq:gke} 
  satisfies H-theorem, 
  then it satisfies GWSC on $(M, g)$. 
\end{theorem}

\begin{proof}
  For each $\tilde{f} \in \tilde{M}$ and 
  $h \in T_{\iota(\tilde{f})} M$, choose 
  a curve $\kappa \mapsto f^\kappa = f + \kappa h$ on $U \cap M$.
  For each $\tilde{h} \in T_{\tilde{f}} \tilde{M}$, 
  \cref{eq:H-bE} yields that 
  \begin{displaymath}
    0 = \od{}{\kappa} \partial^2 H(\tilde{f}) 
    (\tilde{h}, Q[f^\kappa]) \bigg|_{\kappa = 0} = 
    \partial^2 H(\tilde{f}) (\tilde{h}, D_h Q|_{\iota(\tilde{f})}) = 
    g(D_h Q, \iota_* \tilde{h}) |_{\iota(\tilde{f})},
  \end{displaymath}
  which yields that $\iota_* \pi_* D_h Q |_{\iota(\tilde{f})} = 0$.
  By H-theorem, one has that $\partial H(f^0) \in \bE$, 
  $Q[f^0] \equiv 0$, $S(f^0) = 0$, and thus, 
  \begin{displaymath}
    S(f^\kappa) = \inn{\partial H(f^\kappa), Q[f^\kappa]} = 
    \inn{\partial H(f^\kappa) - \partial H(f^0), 
    Q[f^\kappa] - Q[f^0]} \le S(\tilde{f}) = 0.
  \end{displaymath}
  which yields that 
  \begin{displaymath}
    \frac12 
    \dfrac{\rd^2}{\rd \kappa^2} S(f^\kappa) \bigg|_{\kappa = 0} = 
    \partial^2 H(f^0) \big( \partial Q[f^0] (h), h \big) = 
    g(D_h Q, h) |_{\iota(\tilde{f})} \le 0.
  \end{displaymath}
  Therefore, the proof is completed. 
\end{proof}

As for the property of GSC for the kinetic equation \eqref{eq:gke}, 
we need a stronger version of H-theorem.

\begin{definition}
  \label{def:dH}
  We say that the kinetic equation \cref{eq:gke} satisfies 
  \textit{non-degenerate H-theorem} if it satisfies H-theorem and 
  satisfies that for each $\tilde{f} \in \tilde{M}$ and each 
  $h \in T_{\iota(\tilde{u})} M \setminus 
  \iota_* T_{\tilde{u}} \tilde{M}$, there exists a constant 
  $\lambda > 0$ such 
  that for sufficiently small $|\kappa|$, one has that 
  \begin{displaymath}
    S(\tilde{f} + \kappa h) \le - \lambda |\kappa|^2.
  \end{displaymath}
\end{definition}

The following theorem is straightforward according to the proof of 
\Cref{thm:H->GWSC}.

\begin{theorem}
  \label{thm:dH->GSC}
  Suppose that the Riemannian metric $g$ satisfies 
  \Cref{assump:metric}. If the kinetic equation \eqref{eq:gke} 
  satisfies non-degenerate H-theorem, 
  then it satisfies GSC on $(M, g)$.
\end{theorem}

\begin{remark}
  The constant $\lambda$ in \Cref{def:dH} 
  can depend on the point $\tilde{f}$ 
  and the tangent vector $h$, which is often a weak requirement. 
  If there exists a constant $\lambda > 0$ independent of 
  $\tilde{f}$ and $h$ such that 
  \begin{displaymath}
    S(\tilde{f} + \kappa h) \le - \lambda |\kappa|^2 
    \norm{h - \iota_* \pi_* h}^2_{g |_{\iota(\tilde{f})}},
  \end{displaymath}
  then the kinetic equation \eqref{eq:gke} 
  satisfies GUSC on $(M, g)$.
\end{remark}

We can prove the converse part given by the following theorem.

\begin{theorem}\label{thm:GUSC->H}
  Suppose that the Riemannian metric $g$ satisfies 
  \Cref{assump:metric} and that the equilibrium manifold is given by 
  \cref{eq:H-equi}.
  If the kinetic equation \eqref{eq:gke} satisfies the condition 
  \ref{item:uniform_stability} in the definition of GUSC on $(M, g)$, 
  then \cref{eq:gke} satisfies H-theorem. 
\end{theorem}

\begin{proof}
  For each $\tilde{f} \in \tilde{M}$ and 
  each tangent vector $h \in T_{\iota(\tilde{f})} M$ belonging to 
  a neighborhood of the origin at $T_{\iota(\tilde{f})} M$,
  define $S_{\tilde{f}}(h) := S(\tilde{f} + h)$.
  By \cref{eq:H-equi}, one has that 
  \begin{displaymath}
    S_{\tilde{f}}(h) = \inn{\partial H(\tilde{f} + h), 
    Q[\tilde{f} + h]} 
    = \inn{\partial H(\tilde{f} + h) - 
    \partial H(\tilde{f}), 
    Q[\tilde{f} + h] - Q[\tilde{f}]},
  \end{displaymath}
  which yields that
  \begin{displaymath}
    S_{\tilde{f}}(0) = 0, \quad 
    \partial S_{\tilde{f}}(0) = 0.
  \end{displaymath}
  By Taylor formula, for each $\tilde{f} \in \tilde{M}$, there exists
  $\delta(\tilde{f}) > 0$ dependent on $\tilde{f}$ continuously
  such that for each 
  $h \in T_{\iota(\tilde{f})} M$ satisfying that 
  $\| { h } \|_{g|_{\iota(\tilde{f})}}  \le \delta(\tilde{f})$, 
  one has that 
  \begin{displaymath}
    \bigg| S_{\tilde{f}}(h) - 
    \frac{1}{2} \partial^2 S_{\tilde{f}}(0) (h, h) \bigg|
    \le \frac{1}{4} \lambda \norm{h}_{g|_{\iota(\tilde{f})}}^2,
  \end{displaymath}
  where 
  \begin{displaymath}
    \frac{1}{2} \partial^2 S_{\tilde{f}}(0) (h, h) = 
    g(D_h Q, h) |_{\iota(\tilde{f})} \le - \lambda 
    \norm{h - \iota_* \pi_* h}_{g|_{\iota(\tilde{f})}}^2 ,
  \end{displaymath}
  by the condition \ref{item:uniform_stability}.
  Define 
  \begin{displaymath}
    N(\tilde{f}) := \left\{ 
    h \in T_{\iota(\tilde{f})} M \,\Big|\,
    \sqrt{2} \| \iota_* \pi_* h \|_{g|_{\iota(\tilde{f})}}  
    \le \| h \|_{g|_{\iota(\tilde{f})}} \le \delta(\tilde{f})
    \right\}.
  \end{displaymath}
  For each $h \in N(\tilde{f})$, one has that 
  \begin{displaymath}
    S_{\tilde{f}}(h) \le - \lambda 
    \norm{h - \iota_* \pi_* h}_{g|_{\iota(\tilde{f})}}^2  + 
    \frac{1}{4} \lambda \norm{h}_{g|_{\iota(\tilde{f})}}^2 
    \le - \frac{1}{2} \lambda 
    \norm{h - \iota_* \pi_* h}_{g|_{\iota(\tilde{f})}}^2 \le 0,
  \end{displaymath}
  and the equality sign holds if and only if $h = 0$, i.e., 
  $\tilde{f} + h = \tilde{f} \in \tilde{M}$.
  Note that the set $\bigcup_{\tilde{f} \in \tilde{M}} 
  \big( \tilde{f} + N(\tilde{u}) \big)$ 
  contains an open neighborhood $\tilde{U} \subset M$ of 
  the equilibrium manifold $\tilde{M}$.
  Therefore, \cref{eq:gke} satisfies 
  the condition \ref{item:H2} on the open set $\tilde{U}$,
  which completes the proof.
\end{proof}

\begin{remark}
  These results on the relation between the linear stability 
  conditions and the H-theorem can be applied to 
  \cref{eq:1stonM} with a few modifications, 
  assuming that the manifold $M$ is an 
  immersed submanifold in the topological space $V$. For example, 
  one needs to use exponential mappings to replace the term 
  $\tilde{f} + h$ in the proof of \Cref{thm:GUSC->H}.
  The finite-dimensional version of H-theorem was also studied in 
  \cite{Yong2008interesting}.
\end{remark}

Let us consider several examples of kinetic equations to 
illustrate the discussions in this section. 
For each $\tilde{f} \in \tilde{M}$ and a closed subspace $W$ of the 
tangent space $T_{\iota(\tilde{f})} M$, let us denote the 
orthogonal projection from $T_{\iota(\tilde{f})} M$ to $W$ w.r.t. 
the inner product $g|_{\iota(\tilde{f})} = \partial^2 H(\tilde{f})$ 
by $\cP_{W}$.

\begin{example}
  Consider the Boltzmann equation \eqref{eq:boltzmann} with 
  $\+\Xi = \bR^d$ and $v^j(\+\xi) = \xi^j$. 
  Given the distribution function $f = f(\+\xi) \in M$, define 
  the density $\rho$, flow velocity $\+u$, 
  temperature $\theta$, pressure tensor $P = (p^{i,j})_{d \times d}$, 
  and heat flux $\+q$ by 
  \begin{equation}\label{eq:moments}
    \begin{aligned}
      & \rho := \int_{\bR^d} f(\+\xi) \,\rd \+\xi, \qquad
      \rho \+u := \int_{\bR^d} \+\xi f(\+\xi) \,\rd \+\xi, \qquad
      d \rho \theta := \int_{\bR^d} |\+\xi - \+u|^2 f(\+\xi) 
      \,\rd \+\xi, \\
      & \rho P := \int_{\bR^d} (\+\xi - \+u) \otimes (\+\xi - \+u) 
      f(\+\xi) \,\rd \+\xi, \qquad
      \rho \+q := \int_{\bR^d} |\+\xi - \+u|^2 (\+\xi - \+u) f(\+\xi) 
      \,\rd \+\xi.
    \end{aligned}
  \end{equation}
  We will consider four collision terms, all satisfying
  the H-theorem in \Cref{def:Hthm} with the function $H$ in the form 
  of \cref{eq:kinetic_entropy} with 
  $\eta(\+\xi, f) = \eta(f) = f \log f - f$ and the space of 
  collision invariants defined as 
  \begin{equation}\label{eq:Boltzmann_CI}
    \bE = \mathrm{span} 
    \left\{ 1, \xi^1, \ldots, \xi^d, |\+\xi|^2 \right\}.
  \end{equation}
  At this time, the minimizer of the function $H$ 
  under the constraints imposed by 
  $\rho > 0$, $\+u \in \bR^d$ and $\theta > 0$ takes the form of 
  \begin{displaymath}
    f_{\mathrm{eq}}(\+\xi) := \frac{\rho}{(2 \pi \theta)^{d / 2}} 
    \exp \left( - \frac{|\+\xi - \+u|^2}{2 \theta} \right).
  \end{displaymath}

  Firstly, let us consider the two-body collision term given by 
  \cref{eq:boltzmann_twobody}. Note that 
  \begin{multline*}
    S(f) = - \frac14 \int_{\bR^d} \int_{\bR^d} \int_{S^{d-1}} 
    B(|\+\xi - \+\xi_*|, \cos \chi) \\
    \big( f(\+\xi') f(\+\xi_*') - f(\+\xi) f(\+\xi_*) \big)
    \log \frac{f(\+\xi') f(\+\xi_*')}{f(\+\xi) f(\+\xi_*)}
    \,\rd \+\sigma \,\rd \+\xi_* \,\rd \+\xi \le 0.
  \end{multline*}
  Moreover, the equality sign holds if and only if 
  $f(\+\xi) f(\+\xi_*) = f(\+\xi') f(\+\xi_*')$, which is equivalent 
  to that 
  $\eta'(f) = \log f \in \bE$ by \cite[p. 36--42]{Cercignani}.
  Therefore, the Boltzmann equation with the two-body collision term 
  satisfies H-theorem. One can show that it satisfies the 
  non-degenerate 
  H-theorem and thus satisfies GSC. We can also prove it directly. 
  By \Cref{thm:H->GWSC}, it satisfies GWSC.
  For each $\tilde{f} \in \tilde{M}$, one has that 
  $W_0 := \iota_* (T_{\tilde{f}} \tilde{M}) = 
  \tilde{f} \cdot \bE \subset T_{\iota(\tilde{f})} M$, which yields 
  that $\cP_{W_0} w = \iota_* \pi_* w$ for each 
  $w \in T_{\iota(\tilde{f})} M$. By direct calculation, 
  \begin{multline*}
      g(D_w Q, w) = - \frac14 \int_{\bR^d} 
      \int_{\bR^d} \int_{S^{d-1}} B(|\+\xi - \+\xi_*|, \cos \chi) \\
      \qquad \bigg| \frac{w(\+\xi)}{\tilde{f}(\+\xi)} + 
      \frac{w(\+\xi_*)}{\tilde{f}(\+\xi_*)} - 
      \frac{w(\+\xi')}{\tilde{f}(\+\xi')} - 
      \frac{w(\+\xi_*')}{\tilde{f}(\+\xi_*')} \bigg|^2
      \tilde{f}(\+\xi) \tilde{f}(\+\xi_*)
      \,\rd \+\sigma \,\rd \+\xi_* \,\rd \+\xi \le 0.
  \end{multline*}
  Moreover, the equality sign holds if and only if 
  $w \tilde{f}^{-1} \in \bE$, equivalent to that 
  $w \in W_0$, which yields that the Boltzmann equation with the 
  two-body collision 
  term satisfies GSC \cite{Cercignani}. 

  Secondly, consider the Bhatnagar--Gross--Krook (BGK) collision term:
  \begin{displaymath}
    Q[f] := \frac1\tau (f_{\mathrm{eq}} - f),
  \end{displaymath}
  where $\tau > 0$ is the relaxation time.
  For each $\tilde{f} \in \tilde{M}$, one has that 
  $\tilde{f} = \tilde{f}_{\mathrm{eq}}$ and thus,
  $\partial H(\tilde{f}) = \eta'(\tilde{f}) = \log \tilde{f} \in \bE$.
  Since the equality $D_w Q = \frac{1}{\tau} (\cP_{W_0} w - w)$ holds 
  for each $w \in T_{\iota(\tilde{f})} M$, one has that 
  \begin{displaymath}
    g(D_w Q, w) = \frac{1}{\tau} g(\cP_{W_0} w - w, w) 
    = - \frac{1}{\tau} \|{w - \cP_{W_0} w}\|_g^2.
  \end{displaymath}
  Therefore, the BGK model satisfies H-theorem and GUSC by applying 
  \Cref{thm:H->GWSC,thm:GUSC->H}.

  Thirdly, consider the Shakhov collision term:
  \begin{displaymath}
    Q[f] := \frac1\tau (f_{\mathrm{S}} - f), 
  \end{displaymath}
  where $\mathrm{Pr} > 0$ is the Prandtl number, and
  \begin{displaymath}
    f_{\mathrm{S}}(\+\xi) := f_{\mathrm{eq}}(\+\xi) \left( 1 + 
    \frac{(1 - \mathrm{Pr}) \+q^\top (\+\xi - \+u)}{(d + 2) \theta^2}
    \left( \frac{|\+\xi - \+u|^2}{2 \theta} - \frac{d + 2}{2} \right) 
    \right).
  \end{displaymath}
  For each $\tilde{f} \in \tilde{M}$, one has that 
  $\tilde{f} = \tilde{f}_{\mathrm{S}}$ and thus,
  \begin{displaymath}
    \rho \+q = \int_{\bR^d} |\+\xi - \+u|^2 (\+\xi - \+u) 
    \tilde{f}_S(\+\xi) 
    \,\rd \+\xi = (1 - \mathrm{Pr}) \rho \+q,
  \end{displaymath}
  which yields that $\+q = \+0$. Therefore,
  $\log \tilde{f} = \log \tilde{f}_{\mathrm{eq}} \in \bE$.
  Define 
  \begin{displaymath}
    W_1 := \tilde{f} \cdot \mathrm{span} 
    \left\{ (\xi^j - u^j) \big( |\+\xi - \+u|^2 - (d + 2) 
    \theta \big) 
    \right\}_{j = 1}^d \subset 
    T_{\iota(\tilde{f})} M,
  \end{displaymath}
  which is orthogonal to $W_0$ w.r.t. the inner product 
  $g|_{\iota(\tilde{f})}$. One has that $D_w Q = 
  \frac{1}{\tau} (\cP_{W_0} w + (1 - \mathrm{Pr}) \cP_{W_1} w - w)$
  for each $w \in T_{\iota(\tilde{f})} M$, which yields that 
  \begin{multline*}
    g(D_w Q, w) = - \frac{1}{\tau} \|{ w - \cP_{W_0} w }\|_{g}^2 + 
    \frac{1 - \mathrm{Pr}}{\tau} 
    \|{ \cP_{W_1} (w - \cP_{W_0} w) }\|_{g}^2 \\
    \le - \frac{1}{\tau} \min \{ \mathrm{Pr}, 1 \} 
    \|{ w - \cP_{W_0} w }\|_{g}^2.
  \end{multline*}
  Therefore, the Shakhov model satisfies H-theorem \cite{Shakhov} and 
  GUSC \cite{Bae2023Shakhov} 
  by applying \Cref{thm:H->GWSC,thm:GUSC->H}.

  Finally, consider the Ellipsoidal BGK (ES-BGK) collision term:
  \begin{displaymath}
    Q[f] := \frac{\mathrm{Pr}}{\tau} (f_{\mathrm{G}} - f),
  \end{displaymath}
  where $\mathrm{Pr} \ge \frac{d-1}{d} > 0$ is the Prandtl number, 
  \begin{displaymath}
    f_{\mathrm{G}}(\+\xi) := \frac{\rho}{\sqrt{\det(2 \pi \Lambda)}}
    \exp \left( - \frac12 (\+\xi - \+u)^\top \Lambda^{-1} 
    (\+\xi - \+u) \right),
  \end{displaymath}
  and $\Lambda := \frac{1}{\mathrm{Pr}} \theta I + 
  \left( 1 - \frac{1}{\mathrm{Pr}} \right) P$ 
  is a symmetric and positive definite matrix.
  For each $\tilde{f} \in \tilde{M}$, one has that 
  $\tilde{f} = \tilde{f}_{\mathrm{G}}$ and thus, 
  $\Lambda = \theta I$, which yields that 
  $\log \tilde{f} = \log \tilde{f}_{\mathrm{eq}} \in \bE$.
  Define 
  \begin{multline*}
    W_2 := \mathrm{span} 
    \left\{ d |\xi^j - u^j|^2 - |\+\xi - \+u|^2 
    \right\}_{j = 1}^{d - 1} \\ \oplus 
    \left\{ (\xi^{j_1} - u^{j_1}) \cdot (\xi^{j_2} - u^{j_2}) \right\}
    _{1 \le j_1 < j_2 \le d} \subset T_{\iota(\tilde{f})} M,
  \end{multline*}
  which is orthogonal to $W_0$ w.r.t. the inner product 
  $g|_{\iota(\tilde{f})}$. One has that $D_w Q = 
  \frac{\mathrm{Pr}}{\tau} (\cP_{W_0} w + (1 - \mathrm{Pr}^{-1}) 
  \cP_{W_2} w - w)$ for each $w \in T_{\iota(\tilde{f})} M$, 
  which yields that 
  \begin{multline*}
    g(D_w Q, w) = - \frac{\mathrm{Pr}}{\tau} 
    \|{ w - \cP_{W_0} w }\|_{g}^2 + 
    \frac{\mathrm{Pr} - 1}{\tau} 
    \|{ \cP_{W_2} (w - \cP_{W_0} w) }\|_{g}^2  \\
    \le - \frac{1}{\tau} \min \{ \mathrm{Pr}, 1 \} 
    \|{ w - \cP_{W_0} w }\|_{g}^2.
  \end{multline*}
  Therefore, the ES-BGK model satisfies H-theorem 
  \cite{Andries2000ESBGK} and GUSC \cite{Yun2015ESBGK} by 
  applying \Cref{thm:H->GWSC,thm:GUSC->H}.
\end{example}

\begin{example}
  Now let us consider the quantum Boltzmann equation in the form of 
  \cref{eq:gke} with $\+\Xi = \bR^d$ and $v^j(\+\xi) = \xi^j$.

  Firstly, let us consider the Boltzmann--Fermi collision term for 
  fermions:
  \begin{multline}\label{eq:QB_collision}
      Q[f] := \int_{\bR^d} \int_{S^{d-1}} 
      B(|\+\xi - \+\xi_*|, \cos \chi) \Big( f(\+\xi') f(\+\xi_*') 
      \big( 1 + \eps f(\+\xi) \big) \big( 1 + \eps f(\+\xi_*) \big) \\
      \quad - f(\+\xi) f(\+\xi_*) 
      \big( 1 + \eps f(\+\xi') \big) \big( 1 + \eps f(\+\xi_*') \big) 
      \Big) \,\rd \+\sigma \,\rd \+\xi_*,
  \end{multline}
  where $\eps < 0$ is a constant, $\+\xi'$ and $\+\xi_*'$ are 
  given by \cref{eq:collision} as in the classical case.
  The equation is supplemented with the bound 
  $0 \le f \le - 1 / \eps$.
  The Boltzmann--Fermi model satisfies H-theorem with the 
  function $H$ in the form of \cref{eq:kinetic_entropy} with 
  $\eta(\+\xi, f) = \eta(f) = 
  f \log f - \eps^{-1} (1 + \eps f) \log(1 + \eps f)$
  and the space $\bE$ of collision invariants given by 
  \cref{eq:Boltzmann_CI}. 
  Note that 
  \begin{multline*}
      S(f) = - \frac14 \int_{\bR^d} \int_{\bR^d} \int_{S^{d-1}} 
      B(|\+\xi - \+\xi_*|, \cos \chi) \Big( f(\+\xi') f(\+\xi_*') 
      \big( 1 + \eps f(\+\xi) \big) \big( 1 + \eps f(\+\xi_*) \big) \\
      \quad - f(\+\xi) f(\+\xi_*) \big( 1 + \eps f(\+\xi') \big) 
      \big( 1 + \eps f(\+\xi_*') \big) \Big) 
      \log \frac{f(\+\xi') f(\+\xi_*') \big( 1 + \eps f(\+\xi) \big) 
      \big( 1 + \eps f(\+\xi_*) \big)}
      {f(\+\xi) f(\+\xi_*) \big( 1 + \eps f(\+\xi') \big) 
      \big( 1 + \eps f(\+\xi_*') \big)} \\ 
      \,\rd \+\sigma \,\rd \+\xi_* \,\rd \+\xi \le 0,
  \end{multline*}
  and the equality sign holds if and only if 
  \begin{displaymath}
    f(\+\xi') f(\+\xi_*') \big( 1 + \eps f(\+\xi) \big) 
    \big( 1 + \eps f(\+\xi_*) \big) = f(\+\xi) f(\+\xi_*) 
    \big( 1 + \eps f(\+\xi') \big) \big( 1 + \eps f(\+\xi_*') \big),
  \end{displaymath}
  which is equivalent to that
  $\eta'(f) = \log \frac{f}{1 + \eps f} \in \bE$.
  For each $\tilde{f} \in \tilde{M}$, 
  define $\rho$, $\+u$ and $\theta$ as in \cref{eq:moments}
  and define
  $\phi(\+\xi) := \eta'(\tilde{f}(\+\xi)) = \alpha - \beta 
  |\+\xi - \+u|^2$
  with $\alpha \in \bR$ and $\beta > 0$.
  Let $\theta_{\mathrm{c}} := \frac{1}{(d + 2) \pi} 
  ( - \eps \rho \Gamma(1 + d / 2) )^{\frac{2}{d}}$, 
  which is related to the Fermi energy,
  where $\Gamma$ is the gamma function defined by 
  $\Gamma(z) := \int_{0}^{+\infty} e^{-s} s^{z - 1} \,\rd s$ 
  for $\Re(z) > 0$.
  If $\theta > \theta_{\mathrm{c}}$, then 
  $\tilde{f}(\+\xi) = (\exp(- \phi(\+\xi)) - \eps)^{-1}$;
  if $\theta = \theta_{\mathrm{c}}$, then 
  $\tilde{f}(\+\xi) = - \eps^{-1} \mathbbm{1}_{|\+\xi - \+u|^2 < 
  {(d + 2)\theta_{\mathrm{c}}}} (\+\xi)$, which can be obtained 
  by letting $\alpha = (d + 2) \theta_{\mathrm{c}} \beta$ tend to 
  $+ \infty$. Consider the case that $\theta < \theta_c$ when 
  $\tilde{f}$ lies in the interior of $\tilde{M}$. 
  Note that $W_3 := \iota_* (T_{\tilde{f}} \tilde{M}) = 
  \tilde{f} (1 + \eps \tilde{f}) \cdot \bE \subset 
  T_{\iota(\tilde{f})} M$, which yields 
  that $\cP_{W_3} w = \iota_* \pi_* w$ for each 
  $w \in T_{\iota(\tilde{f})} M$. 
  By direct calculation, 
  \begin{multline*}
    g(D_w Q, w) = \\
    - \frac14 \int_{\bR^d} \int_{\bR^d} \int_{S^{d-1}} 
    B(|\+\xi - \+\xi_*|, \cos \chi) \bigg| 
    \frac{w(\+\xi)}{\tilde{f}(\+\xi) (1 + \eps \tilde{f}(\+\xi))} + 
    \frac{w(\+\xi_*)}{\tilde{f}(\+\xi_*) 
    (1 + \eps \tilde{f}(\+\xi_*))}
    \\ -
    \frac{w(\+\xi')}{\tilde{f}(\+\xi') (1 + \eps \tilde{f}(\+\xi'))} -
    \frac{w(\+\xi_*')}{\tilde{f}(\+\xi_*') 
    (1 + \eps \tilde{f}(\+\xi_*'))}
    \bigg|^2 \tilde{f}(\+\xi) \tilde{f}(\+\xi_*) 
    (1 + \eps \tilde{f}(\+\xi')) (1 + \eps \tilde{f}(\+\xi_*')) \\
    \,\rd \+\sigma \,\rd \+\xi_* \,\rd \+\xi \le 0,
  \end{multline*}
  and the equality sign holds if and only if 
  $w \tilde{f}^{-1} (1 + \eps \tilde{f})^{-1} \in \bE$, 
  which is equivalent to that $w \in W_3$.
  This yields that the Boltzmann--Fermi model satisfies GSC.

  The Boltzmann--Bose collision term for bosons is the same as 
  \cref{eq:QB_collision} but with $\eps > 0$.
  According to the argument of the above example, the 
  Boltzmann--Bose model also satisfies H-theorem.
  For each $\tilde{f} \in \tilde{M}$, define 
  $\phi(\+\xi) := \eta'(\tilde{f}(\+\xi)) = \alpha - \beta 
  |\+\xi - \+u|^2$
  with $\alpha \le - \log \eps$ and $\beta > 0$.
  Define the critical temperature of Bose-Einstein condensation 
  by $\theta_{\mathrm{c}} := (2 \pi d)^{-1} Z(\frac{d}{2} + 1) 
  Z(\frac{d}{2})^{- \frac{2}{d} - 1} (\eps \rho)^{\frac{2}{d}}$
  with $d \ge 2$, where $Z$ is the Riemann zeta function defined by 
  $Z(z) := \sum_{n = 1}^{\infty} n^{-z}$ with $\Re(z) > 1$ and 
  $Z(1) := + \infty$. If $\theta > \theta_{\mathrm{c}}$, then 
  $\tilde{f}(\+\xi) = (\exp(- \phi(\+\xi)) - \eps)^{-1}$ with 
  $\alpha < - \log \eps$; if $\theta \le \theta_{\mathrm{c}}$, then 
  $\tilde{f}(\+\xi) = (\exp(- \phi(\+\xi)) - \eps)^{-1} + \gamma 
  \delta(\+\xi - \+u)$ with $\alpha = - \log \eps$, 
  $\beta^{-1} = 2 d Z(\frac{d}{2}) Z(\frac{d}{2} + 1)^{-1} \theta$
  and $\gamma = \rho - \eps^{-1} \beta^{- \frac{d}{2}} 
  \pi^{\frac{d}{2}} Z(\frac{d}{2}) \ge 0$, 
  where $\delta$ is Dirac delta function.
  One can obtain that the Boltzmann--Bose model also satisfies GSC 
  using the same argument as the previous case.
\end{example}

\subsection{Natural model reduction}
In this part, let us review the motivations and methodologies of 
natural model reduction proposed in \Cref{sec:reduction} 
for the kinetic equation \eqref{eq:gke} and 
demonstrate how the conditions that guarantee specific
structure-preserving properties of reduced models can be verified or 
satisfied. 
Furthermore, we give a posteriori error estimate for the model 
reduction method.

\subsubsection{Motivations and methodologies}
To carry out the model reduction in the ordinate variables $\+\xi$, 
we aim to reduce the solution functions to be defined on 
$\bR \times \Omega \times \{1, \cdots, n\}$
rather than on $\bR \times \Omega \times \+\Xi$, 
where $n$ is a positive integer. 
Suppose that the ansatz for the phase density functions 
takes the form of 
\begin{equation}\label{eq:kinetic_ansatz}
  \hat{f}(\+\xi; t, \+x) = \hat{f}(\+\xi; \+\omega(t, \+x)),
\end{equation}
where the parameters are $\+\omega : \bR \times \Omega \to \bR^n$.
Therefore, the ansatz manifold $\hat{M}$ can be defined as follows,
\begin{displaymath}
  \hat{M} = \left\{ \hat{f} = \hat{f}(\cdot; \+\omega) : 
  \+\Xi \to \bR \,\big|\, 
  \+\omega = (\omega^1, \omega^2, \ldots, \omega^n)^\top \in \bR^n 
  \right\}.
\end{displaymath}
The tangent space $T_{\hat{f}} \hat{M}$ of the manifold $\hat{M}$ at 
$\hat{f}(\cdot; \+\omega)$ is 
spanned by the partial derivatives of $\hat{f}$ w.r.t. the parameters 
$\omega^k$, where $k = 1, 2, \ldots, n$.

Model reduction aims to derive the governing equations for 
$\+\omega$, which should satisfy a system of PDEs in $t$ and $\+x$.
Let us revisit the motivations and 
methodologies of model reduction in the context of the kinetic 
equations given by \cref{eq:gke}. By substituting the ansatz 
\eqref{eq:kinetic_ansatz} into the kinetic equation \eqref{eq:gke}, 
a problem arises that the distribution function $\hat{f}$ does not 
satisfy \cref{eq:gke} since the time derivative   
$\tpd{\hat{f}}{t} = \tpd{\omega^k}{t} \tpd{\hat{f}}{\omega^k}$
lies in the tangent space $T_{\hat{f}} \hat{M}$ while the other term 
$Q[\hat{f}] - \+v(\+\xi) \cdot \nabla_{\+x} \hat{f} = 
Q[\hat{f}] - v^j(\+\xi) \tpd{\omega^\ell}{x^j} 
\tpd{\hat{f}}{\omega^\ell}$ does not necessarily. 
Therefore, it is natural to carry out 
model reduction by using a projection $\bP_{\+\omega}$ from the space 
$V$ to the subspace $T_{\hat{f}} \hat{M}$,
where $\hat{f} = \hat{f}(\cdot; \+\omega)$, and then projecting 
$v^j(\+\xi) \tfrac{\partial \hat{f}}{\partial \omega^\ell}$ and 
$Q[\hat{f}]$ onto $T_{\hat{f}} \hat{M}$. Suppose that 
\begin{displaymath}
  \bP_{\+\omega} \left( v^j(\+\xi) \pd{\hat{f}}{\omega^\ell} \right) 
  = m^{j,k}_{\ell}(\+\omega) \pd{\hat{f}}{\omega^k},
  \quad \bP_{\+\omega} Q[\hat{f}] 
  = q^k(\+\omega) \pd{\hat{f}}{\omega^k}.
\end{displaymath}
Therefore, a reduced model can be given as follows,
\begin{displaymath}
  \pd{\omega^k}{t} + m_\ell^{j,k}(\+\omega) 
  \pd{\omega^\ell}{x^j} = q^k(\+\omega), \quad 
  k = 1, 2, \ldots, n.
\end{displaymath}

\subsubsection{Strucutral properties}
To ensure the hyperbolicity and linear stability of the reduced model 
and to make the projection uniquely determined, we need 
to specify the orthogonality of the projection $\bP_{\+\omega}$ 
w.r.t. the inner product $g|_{i(\hat{f})}$, where 
$\hat{f} = \hat{f}(\cdot; \+\omega)$.
By \Cref{thm:bilinear_hyperbolicity}, the kinetic equation 
\eqref{eq:gke} is hyperbolic w.r.t. the Riemannian metric 
$g|_f = \inn{\cdot, \cdot}_{L^2(\+\Xi, \mu^f)}$, where $\mu^f$ is a 
strictly positive Radon measure on the manifold 
$\+\Xi$ for each $f \in M$.
According to the discussions in \Cref{sec:kinetic_linear}, the 
Riemannian metric $g$ should also satisfy \cref{eq:kinetic_natural},
which determines the measure $\mu^{\tilde{f}}$ for each 
$\tilde{f} \in \tilde{M}$ the ansatz manifold $\hat{M}$ contains the 
equilibrium manifold $\tilde{M}$,
and the projection $\bP_{\+\omega}$ should satisfy that 
$\bP_{\+\omega} Q[\hat{f}(\cdot; \+\omega)] = 0$ if and only if 
$\hat{f} \in \tilde{M}$, to guarantee linear stability.
The reduced model is a symmetric hyperbolic system 
given by \cref{eq:reduce_compact} as follows:
\begin{displaymath}
  \bigg\langle \pd{\hat{f}}{\omega^k}, 
  \pd{\hat{f}}{\omega^{\ell}} 
  \bigg\rangle _{L^2(\+\Xi, \mu^{\hat{f}})} 
  \pd{\omega^{\ell}}{t} + 
  \bigg\langle \pd{\hat{f}}{\omega^k}, v^j 
  \pd{\hat{f}}{\omega^{\ell}} 
  \bigg\rangle _{L^2(\+\Xi, \mu^{\hat{f}})} 
  \pd{\omega^{\ell}}{x^j} = 
  \bigg\langle \pd{\hat{f}}{\omega^k}, 
  \hat{Q} |_{\hat{f}(\+\omega)} 
  \bigg\rangle _{L^2(\+\Xi, \mu^{\hat{f}})}.
\end{displaymath}
By \Cref{thm:reduce_speed}, the reduced model preserves the property 
of finite propagation speed naturally.
For a given quantity 
$c \in C^\infty(M)$, we can refer to \Cref{rmk:reduce_flux} to 
establish that the reduced models preserve the property of $c(f)$ as 
a quantity with flux, a conserved quantity, or an entropy, given the 
existence of a tangent vector $h \in T_{\hat{f}} \hat{M}$ such that 
$\partial c(\hat{f}; \cdot) = 
\inn{h, \cdot}_{L^2(\+\Xi, \mu^{\hat{f}})}$.

\subsubsection{Error estimate}
Let us quantify the approximation error between the 
approximate solution $\hat{f}(\cdot; \+\omega)$ and the exact
solution $f$ to \cref{eq:gke}. 
Our results show that the error can be bounded 
by the accumulation of the approximation residual of tangent vectors 
under appropriate Lipschitz assumptions of the collision terms. Let
$\delta f := \hat{f}(\cdot; \+\omega) - f$ represent the difference 
between the solutions. 
The approximation residual of tangent vectors can be expressed as 
follows,
\begin{displaymath}
  \begin{aligned}
    & \cR[\hat{f}(\cdot; \+\omega)] = \left( \pd{}{t} + \+v(\+\xi) 
    \cdot \nabla_{\+x} \right) \hat{f}(\+\xi; \+\omega) - 
    Q[\hat{f}(\cdot; \+\omega)] \\
    =\ & (\mathrm{id} - \bP_{\+\omega}) \left( \+v(\+\xi) \cdot 
    \nabla_{\+x} \hat{f}(\+\xi; \+\omega) - 
    Q[\hat{f}(\cdot; \+\omega)] \right) 
  \end{aligned}
\end{displaymath}
Suppose that the solutions decay sufficiently fast at infinity to 
avoid boundary terms.
We will consider the error in the norm defined as,
\begin{displaymath}
  \norm{\delta f}_p = \left( \int_{\bR^d} \int_{\+\Xi} |\delta f|^p \,
  \rd \mu(\+\xi) \,\rd \+x \right)^{1 / p}, \quad p \in (1, +\infty).
\end{displaymath}
The following theorem provides a posteriori error estimate. 
\begin{theorem}
  Suppose that for a given $p \in (1, \infty)$, the collisional
  operator $Q$ satisfies the following Lipschitz-type condition,
  \begin{displaymath}
    \int_{\+\Xi} |f_1 - f_2|^{p-1} |Q[f_1] - Q[f_2]| \,\rd \mu(\+\xi) 
    \leq L_Q
    \int_{\+\Xi} |f_1 - f_2|^p \,\rd \mu(\+\xi), \quad \forall f_1, 
    f_2 \in M,
  \end{displaymath}
  where $L_Q$ is a positive constant. The following error estimate 
  holds:
  \begin{equation}\label{eq:error}
    \| \delta f \|_p (T) \le \| \delta f \|_p (0) + \int_0^T
    e^{L_Q (T - \tau)} 
    \| \cR[\hat{f}(\cdot; \+\omega(\tau, \cdot))] \|_p 
    \,\rd \tau.
  \end{equation}
\end{theorem}

\begin{proof}
  Note that the exact solution $f$ and the approximate solution 
  $\hat{f}(\cdot; \+\omega)$ satisfy the following equations, 
  \begin{displaymath}
    \begin{aligned}
      & \pd{f}{t} + \+v(\+\xi) \cdot \nabla_{\+x} f = Q[f], \\
      & \pd{\hat{f}}{t} + 
      \+v(\+\xi) \cdot \nabla_{\+x} \hat{f}(\+\xi; \+\omega) = 
      Q[\hat{f}(\cdot; \+\omega)] + \cR[\hat{f}(\cdot; \+\omega)],
    \end{aligned}
  \end{displaymath}
  respectively. Therefore, 
  \begin{displaymath}
    \pd{(\delta f)}{t} + \+v(\+\xi) \cdot \nabla_{\+x} (\delta f) = 
    \delta Q + \cR[\hat{f}(\cdot; \+\omega)],
  \end{displaymath}
  where $\delta Q = Q[\hat{f}(\cdot; \+\omega)] - Q[f]$.
  One has that 
  \begin{multline*}
    \frac{1}{p} \od{}{t} \int_{\bR^d} \int_{\+\Xi} 
    |\delta f|^p \,\rd \mu(\+\xi) \,\rd \+x + 
    \frac{1}{p} \int_{\bR^d} \int_{\+\Xi} 
    \+v(\+\xi) \cdot \nabla_{\+x} |\delta f|^p 
    \,\rd \mu(\+\xi) \,\rd x \\
    = \int_{\bR^d} \int_{\+\Xi} 
    \left( \cR[\hat{f}(\cdot; \+\omega)] + \delta Q \right) 
    |\delta f|^{p - 2} \delta f \,\rd \mu(\+\xi) \,\rd \+x,
  \end{multline*}
  by multiplying by $|\delta f|^{p - 2} \delta f$ and integrating 
  over $\+\Xi \times \bR^d$.
  The assumption that the solutions decay sufficiently fast at 
  infinity yields that
  \begin{displaymath}
    \int_{\bR^d} \int_{\+\Xi} \+v(\+\xi) \cdot \nabla_{\+x} 
    |\delta f|^p \,\rd \mu(\+\xi) \,\rd x = 0.
  \end{displaymath}
  By H\"older inequality, one has that 
  \begin{displaymath}
    \int_{\bR^d} \int_{\+\Xi} \cR[\hat{f}(\cdot; \+\omega)] 
    |\delta f|^{p-2} \delta f \,\rd \mu(\+\xi) \,\rd \+x \le 
    \|{\cR[\hat{f}(\cdot; \+\omega)]}\|_p \norm{\delta f}_p^{p-1}.
  \end{displaymath}
  The Lipschitz-type condition of the collision operator 
  gives us that 
  \begin{displaymath}
    \int_{\bR^d} \int_{\+\Xi} |\delta f|^{p - 2} 
    \delta f \cdot \delta Q \,\rd \mu(\+\xi) \,\rd \+x \le 
    \int_{\bR^d} \int_{\+\Xi} |\delta f|^{p - 1} \cdot |\delta Q| 
    \,\rd \mu(\+\xi) \,\rd \+x \le L_Q \norm{\delta f}_p^p.
  \end{displaymath}
  Collecting these inequalities together, one has that 
  \begin{displaymath}
    \od{\norm{\delta f}_p}{t} \leq L_Q \norm{\delta f}_p + \|{
    \cR[\hat{f}(\cdot; \+\omega)]}\|_p,
  \end{displaymath}
  and then by Gr\"onwall's inequality, one can obtain 
  \cref{eq:error}. 
\end{proof}

\begin{remark}
  The first term of the error bound \eqref{eq:error}, i.e.,
  $\|\delta f\|_p(0)$, is the approximation error of the initial 
  value, which is known once given the problem under consideration. 
  The second term of the error bound \eqref{eq:error}, in the form of 
  integration, is computable once given the solution $\+\omega$ to 
  the reduced model. Therefore, the error bound \eqref{eq:error} is a 
  posteriori error estimate. It allows us to carry out 
  model reduction without any a priori knowledge of the solution 
  manifold while the error of the reduced model is still
  available. Therefore, in principle, one may adopt the method 
  developed here without assuming that the solution 
  manifold is low-dimensional.
\end{remark}

\subsection{Examples}
\label{sec:example}
We will present some examples showcasing the 
application of model reduction techniques to kinetic equations within 
the proposed framework. These examples are preliminary 
demonstrations intended to provide an initial glimpse into the 
potential of this methodology.

\subsubsection{Entropy minimization principle}

Consider the kinetic equation \eqref{eq:gke} that has an entropy 
$H(f) = \int_{\+\Xi} \eta(f(\+\xi)) \,\rd \mu(\+\xi)$ according to 
\Cref{def:kinetic_entropy}, where $\eta$ is a smooth, strictly convex 
function, and $\mu$ is a strictly positive Radon measure on the 
manifold $\+\Xi$.
According to the entropy minimization principle in \cite{Levermore}, 
let the ansatz manifold $\hat{M}$ be the set of all the 
feasible solutions of the following optimization problem
\begin{displaymath}
	\begin{aligned}
		\min_{f \in M} \ \ & H(f), \\ \text{s.t.} \ \ & 
		\inn{\fc_p, f} := 
    \int_{\+\Xi} m_p(\+\xi) f(\+\xi) \,\rd \mu(\+\xi) = 
		c_p \in \bR, \quad p = 1, 2, \ldots, n,
	\end{aligned}
\end{displaymath}
where $m_p$ is an appropriate function defined on $\+\Xi$ for each 
$p = 1, 2, \ldots, n$ such that the space $\bE$ of collision 
invariants satisfies 
\begin{equation}\label{eq:minimizeEntropyBE}
  \bE \subset 
  \mathrm{span} \left\{ \fc_p \,\big|\, p = 1, 2, \ldots, n \right\}.
\end{equation}
For example, in the Boltzmann equation \eqref{eq:boltzmann}, the 
corresponding optimization problem can be chosen as 
\begin{displaymath}
	\begin{aligned}
		\min_{f \in M} \ \ & 
    H(f) = \int_{\bR^d} f \log f - f \,\rd \+\xi, \\ 
		\text{s.t.} \ \ & 
		\int_{\bR^d} \+\xi^{\+p} f(\+\xi) \,\rd \+\xi = c_{\+p} 
		\in \bR, \quad \+p \in \bN^d, \ |\+p| \le N.
	\end{aligned}
\end{displaymath}
By Karush--Kuhn--Tucker conditions, there exist $n$ constants 
$\alpha^1, \alpha^2, \ldots, \alpha^n$ such that the optimal solution 
$\hat{f}$ satisfies that 
\begin{displaymath}
	\eta'(\hat{f}(\+\xi)) = \alpha^p m_p(\+\xi).
\end{displaymath}
Denote the Legendre transform of $\eta$ as 
$\zeta(z) := \sup_y \big( y z - \eta(y) \big)$. Therefore, 
the ansatz manifold $\hat{M}$ is 
\begin{displaymath}
	\hat{M} = \left\{ \hat{f}(\+\xi) = 
	\zeta' \left( \alpha^p m_p(\+\xi) \right) \,\Bigg|\,
	\alpha^k \in \bR, m_k \hat{f} \in L^1(\+\Xi), 
	\forall k = 1, 2, \ldots, n \right\}.
\end{displaymath}
The tangent space $T_{\hat{f}} \hat{M}$ of $\hat{M}$ at 
$\hat{f} = \zeta' \left( \alpha^p m_p \right)$ is 
\begin{displaymath}
	T_{\hat{f}} \hat{M} = \mathrm{span} \left\{ m_k 
	\zeta'' \left( \alpha^p m_p \right) \,\Bigg|\, 
	k = 1, 2, \ldots, n \right\}.
\end{displaymath}
As shown in \Cref{sec:kinetic_linear},
a natural choice of Riemannian metric is 
\begin{displaymath}
	g|_{i(\hat{f})} = 
	\inn{\cdot, \cdot}_{L^2(\+\Xi, \eta''(\hat{f}(\+\xi)) 
	\,\rd \mu(\+\xi))}.
\end{displaymath}
Using the orthogonal projection w.r.t. the Riemannian 
metric $g$, the model reduction preserves the properties of 
hyperbolicity and linear stability of the kinetic equation. 
The reduced model can be written in 
the compact form \eqref{eq:reduce_compact}, where 
\begin{gather*}
	\mathbf{A}_{k, \ell}^0 = 
	\hat{g}(\tfrac{\partial \hat{f}}{\partial \alpha^k}, 
	\tfrac{\partial \hat{f}}{\partial \alpha^\ell}) = \int_{\+\Xi}
	m_k(\+\xi) m_\ell(\+\xi) \zeta'' \left( \alpha^p m_p(\+\xi) \right)
	\,\rd \mu(\+\xi), \\
	\mathbf{A}_{k, \ell}^j = 
	\hat{g}(\tfrac{\partial \hat{f}}{\partial \alpha^k}, \hat{A}^j
	\tfrac{\partial \hat{f}}{\partial \alpha^\ell}) = 
	\int_{\+\Xi} v^j(\+\xi)
	m_k(\+\xi) m_\ell(\+\xi) \zeta'' \left( \alpha^p m_p(\+\xi) \right)
	\,\rd \mu(\+\xi), \\
	\mathbf{Q}_{k} = 
	\hat{g}(\tfrac{\partial \hat{f}}{\partial \alpha^k}, Q[\hat{f}]) = 
	\int_{\+\Xi} m_k(\+\xi) 
	Q \left[\zeta' \left( \alpha^p m_p(\+\xi) \right) \right] 
	\,\rd \mu(\+\xi).
\end{gather*}
The reduced model can also be written in the conservative form, 
\begin{displaymath}
	\pd{\hat{c}_k}{t} + \pd{\hat{F}^j_k}{x^j} = \mathbf{Q}_k, \quad 
	k = 1, 2, \ldots, n,
\end{displaymath}
where 
\begin{gather*}
	c_k(f) = \int_{\+\Xi} m_k(\+\xi) f(\+\xi) \,\rd \mu(\+\xi), \quad 
	F^j_k(f) = \int_{\+\Xi} v^j(\+\xi) m_k(\+\xi) f(\+\xi) 
	\,\rd \mu(\+\xi), \\
	\hat{c}_k(\hat{f}) = \int_{\+\Xi} m_k(\+\xi) 
	\zeta'(\alpha^p m_p(\+\xi)) \,\rd \mu(\+\xi), \quad 
	\hat{F}^j_k(\hat{f}) = \int_{\+\Xi} v^j(\+\xi) 
	m_k(\+\xi) \zeta'(\alpha^p m_p(\+\xi)) \,\rd \mu(\+\xi).
\end{gather*}
Here $\+F_k$ is the flux of $c_k$ for \cref{eq:gke} 
according to \Cref{def:flux}, which is preserved by the model 
reduction by \Cref{thm:reduce_flux} and \Cref{rmk:reduce_flux}, 
since 
\begin{displaymath}
	\inn{\rd c_k |_{i(\hat{f})}, h} = 
	\int_{\+\Xi} m_k(\+\xi) h(\+\xi) \,\rd \mu(\+\xi) = 
	g \big( m_k \zeta'' \left( \alpha^p m_p \right),
	h \big) \big|_{i(\hat{f})},
\end{displaymath}
for each $h \in T_{i(\hat{f})} M$, where 
$m_k \zeta'' \left( \alpha^p m_p \right) \in T_{\hat{f}} \hat{M}$.
By \cref{eq:minimizeEntropyBE}, the reduced model preserves all the 
conserved quantities of the kinetic equation.
Furthermore, the entropy $H(f)$ is also preserved by the model 
reduction since 
\begin{displaymath}
	\rd H |_{i(\hat{f})} = \alpha^k \rd \hat{c}_k |_{i(\hat{f})} = 
	g \big( \alpha^k m_k \zeta'' \left( \alpha^p m_p \right),
	\cdot \big) \big|_{i(\hat{f})},
\end{displaymath}
and $\alpha^k m_k \zeta'' \left( \alpha^p m_p \right) 
\in T_{\hat{f}} \hat{M}$. 
However, this model is known to be 
impractical for numerical simulations since 
the coefficients in \cref{eq:reduce_compact} are typically expensive 
to compute, and it is challenging to obtain the bijection between the 
coefficients $\alpha^p$ and the moments $c_p$, which involves 
solving a nonlinear equation with a large condition number.

\subsubsection{Polynomial multiplicative perturbation around Maxwellian}
We would like to consider the Boltzmann equation 
\eqref{eq:boltzmann}.
As mentioned previously, the Boltzmann equation has an entropy 
$H(f) = \int_{\bR^d} \eta(f(\+\xi)) \,\rd \+\xi$ where 
$\eta(f) = f \log f - f$. The celebrated H-theorem says 
\begin{displaymath}
	\int_{\bR^d} Q[f] \eta'(f(\+\xi))  \,\rd \+\xi = 
	\int_{\bR^d} Q[f] \log f \,\rd \+\xi \le 0, 
\end{displaymath}
and the equality sign holds if and only if $Q[f] = 0$, or 
equivalently, the function $f$ is a Maxwellian.
The collision term tends to push the solutions toward 
the Maxwellian distribution.
Let the ansatz manifold $\hat{M}$ be the set of distribution 
functions $\hat{f}$, which are polynomial multiplicative 
perturbations 
around local Maxwellian with the same flow velocity $\+u$ and 
temperature $\theta$ as the distribution function $\hat{f}$. 
In other words, 
\begin{multline*}
	\hat{M} = \Bigg\{ \hat{f}(\+\xi) = 
	\frac{\rho}{(2 \pi \theta)^{d / 2}} 
  \exp \left( - \frac{|\+\xi - \+u|^2}{2 \theta} \right) 
	\sum_{\+k \in \bN^d, |\+k| \le N} \hat\alpha^{\+k} \+\xi^{\+k} 
	\,\Bigg|\, \\
	\int_{\bR^d} \hat{f}(\+\xi) \,\rd \+\xi = \rho, 
	\int_{\bR^d} \+\xi \hat{f}(\+\xi) \,\rd \+\xi = \rho \+u, 
	\int_{\bR^d} |\+\xi - \+u|^2 \hat{f}(\+\xi) \,\rd \+\xi = 
	d \rho \theta \Bigg\}.
\end{multline*}
Such ansatz can also be found in \cite{Fan_new,Fan2015}, 
but yields a different reduced model. 
The key difference is which vector bundle of the ansatz 
manifold $\hat{M}$ is projected onto, and we will discuss the 
connections between these models in detail in our future work. 
Define an auxiliary variable 
$\+w = \tfrac{\+\xi - \+u}{\sqrt{\theta}}$ 
and $d$-dimensional Hermite polynomials $H_{\+k}(\+w)$ 
satisfying that 
\begin{displaymath}
	\int_{\bR^d} H_{\+k_1}(\+w) H_{\+k_2}(\+w) \frac{1}{(2 \pi)^{d/2}}
	\exp \left( -\frac12 |\+w|^2 \right) \,\rd \+w = \+k_1 ! 
	\delta_{\+k_1, \+k_2}.
\end{displaymath}
Let $P(\+w) = \tfrac{1}{(2 \pi)^{d/2}} 
\exp \left( -\tfrac12 |\+w|^2 \right) 
\sum_{|\+k| \le N} \alpha^{\+k} H_{\+k}(\+w)$, and 
then the manifold $\hat{M}$ can be written as 
\begin{displaymath}
	\hat{M} = \left\{ \hat{f}(\+\xi) = \frac{\rho}{\theta^{d / 2}} 
	P(\+w) \,\Bigg|\, \alpha^{\+0} = 1,
	\alpha^{\+e_1} = \cdots = \alpha^{\+e_d} = 
	\sum_{k = 1}^d \alpha^{2 \+e_k} = 0 \right\}.
\end{displaymath}
Therefore, the tangent space of $\hat{M}$ at $\hat{f}$ is 
\begin{multline*}
	T_{\hat{f}} \hat{M} = \mathrm{span} \left\{ \tpd{\hat{f}}{\rho}; 
	\tpd{\hat{f}}{u^j}; \tpd{\hat{f}}{\theta} ;
	\tpd{\hat{f}}{\alpha^{2 \+e_1}}; \ldots; 
  \tpd{\hat{f}}{\alpha^{2 \+e_{d-1}}};
	\tpd{\hat{f}}{\alpha^{\+k}}, 2 \le |\+k| \le N, 
	\+k \ne 2 \+e_j \right\} \\
	= \mathrm{span} \bigg\{ P(\+w); \tpd{P(\+w)}{w^j}; 
	\+w \cdot \nabla_{\+w} P(\+w); 
	\exp \left(- \tfrac12 |\+w|^2 \right) 
	\left( H_{2\+e_1}(\+w) - H_{2\+e_d}(\+w) \right); \\ 
	\ldots; \exp \left(- \tfrac12 |\+w|^2 \right) 
	\left( H_{2\+e_{d-1}}(\+w) - H_{2\+e_d}(\+w) \right); \\
	\exp \left(- \tfrac12 |\+w|^2 \right) H_{\+k}(\+w), 
	2 \le |\+k| \le N, \+k \ne 2 \+e_j \bigg\} \\
	= \exp \left(- \tfrac12 |\+w|^2 \right) \cdot 
	\mathrm{span} \Bigg\{ 1; w^j + \sum_{|\+k| = N} 
	\alpha^{\+k} H_{\+k + \+e_j}(\+w); 
	|\+w|^2 + \sum_{|\+k| = N} \sum_{j = 1}^d
	\alpha^{\+k} w^j H_{\+k + \+e_j}(\+w); \\
	H_{2\+e_1}(\+w) - H_{2\+e_d}(\+w); \ldots; 
	H_{2\+e_{d-1}}(\+w) - H_{2\+e_d}(\+w); 
	H_{\+k}(\+w), 2 \le |\+k| \le N, \+k \ne 2 \+e_j
	\Bigg\}.
\end{multline*}
The Riemannian metric $g$ is chosen as 
\begin{displaymath}
	g |_{i(\hat{f})} = \inn{\cdot, \cdot}_{L^2(\+\Xi, 
	\exp (\tfrac12 |\+w|^2) \,\rd \mu(\+\xi))}.
\end{displaymath}
The reduced model satisfies conservation of mass since 
$\exp(-\tfrac12 |\+w|^2) \in T_{\hat{f}} \hat{M}$. However, the 
reduced model does not satisfy conservation of momentum and energy
since $w^j \exp(-\tfrac12 |\+w|^2)$ and 
$|\+w|^2 \exp(-\tfrac12 |\+w|^2)$ do not necessarily lie in the 
tangent space $T_{\hat{f}} \hat{M}$.

\subsubsection{A conservative one-dimensional moment model}
Let us consider the one-dimensional Boltzmann equation. 
We would like to derive
a reduced moment model, which satisfies conservation of mass, 
momentum, and energy, and can be written as a hyperbolic system of 
balanced laws. 

Let the ansatz manifold $\hat{M}$ be 
\begin{multline*}
	\hat{M} = \Bigg\{ \hat{f}(\xi) = 
	\exp \left( - \frac{|\xi - u|^2}{2 \theta} \right) 
	\sum_{k = 0}^N \alpha^k 
	\xi^k \,\Bigg|\, \\ u \in \bR ; \theta > 0 ; \alpha^k \in \bR, 
	k = 0, 1, \ldots, N \Bigg\},
\end{multline*}
which is an $(N + 3)$-dimensional manifold.
We remove the previous requirement on the momentum and energy 
of the Maxwellian in the definition of $\hat{M}$. The tangent space 
is 
\begin{displaymath}
	T_{\hat{f}} \hat{M} = \exp \left( - \frac{|\xi - u|^2}{2 \theta} 
	\right) \cdot \mathrm{span} \left\{ \xi^k \,\big|\, 
	k = 0, 1, \ldots, N + 2 \right\}.
\end{displaymath}
The Riemannian metric $g$ is chosen as 
\begin{displaymath}
	g |_{i(\hat{f})} = \inn{\cdot, \cdot}_{L^2(\+\Xi, 
	\exp (\tfrac{|\xi - u|^2}{2 \theta}) \,\rd \mu(\+\xi))}.
\end{displaymath}
Therefore, the quantities 
$c_k(\hat{f}) = \int_{\bR} \xi^k \hat{f}(\xi) \,\rd \xi$ 
with $k = 0, 1, \ldots, N + 2$ have fluxes for the reduced system.
The corresponding fluxes are 
$F_k(\hat{f}) = \int_{\bR} \xi^{k+1} \hat{f}(\xi) \,\rd \xi$.
With a bijection between the coefficients involved in the ansatz and 
the moments $c_k$, the reduced model can be written as a 
hyperbolic system of balanced laws w.r.t. the moments $c_k$.
The coefficients of \cref{eq:reduce_compact} here can be 
written explicitly.

\section{Conclusions}
\label{sec:concl}
Our framework of model reduction 
employs the projection onto the tangent bundle to reduce 
kinetic equations while preserving fundamental properties, including
hyperbolicity, conservation laws, entropy dissipation, finite 
propagation speed, and linear stability. 
We investigate these properties of kinetic equations within the novel 
framework, determining the choice of Riemannian metric.
This framework unifies the studies on kinetic equations and reduced 
models, provides a theoretical guarantee for model reduction 
from a general perspective,
and applies to reducing various kinetic models.
This framework solves the problems regarding 
the theoretical analysis of 
properties of many reduced models of kinetic equations at once.
The approach presented in this paper holds potential for application in various high-dimensional problems.

The framework proposed in this paper can give a reduced model once 
given the ansatz manifold and a Riemannian metric. Furthermore, we 
can also approximate the tangent bundle by a vector bundle.
It should be noted that the currently presented framework cannot 
include some of the reduced models we mentioned in the introduction, 
such as HME and QMOM, which can be 
included in a broader framework that utilizes the projection onto 
the vector bundle, which approximates the tangent bundle. 
The approach demonstrated in this paper can be 
generalized to the analysis of the properties of this broader 
framework.

Future research directions for this model reduction framework 
encompass both theoretical and practical avenues. On the theoretical 
front, exploring the preservation of other structures, such as 
nonlinear stability around equilibrium states, formulating boundary 
conditions, and (generalized) Kreiss conditions arising in the 
boundary value problems, offers promising prospects. On the 
side of applications, leveraging insights into distribution function 
solutions to kinetic equations in specific problems or physical 
scenarios opens doors to proposing improved reduced models that may 
yield better numerical simulations. Additionally, 
innovative machine learning techniques can provide the ansatz 
manifold in this model reduction framework, potentially 
yielding enhanced outcomes for specific problems.

\appendix

\section{Proof of \Cref{thm:bilinear}}
\label{app:kinetic_properties}
Let us first introduce Nachbin's theorem, which is a variant of 
Stone--Weierstrass theorem for $C^\infty (\Xi)$.

\begin{lemma}[Nachbin's theorem]
  \label{lemma:nachbin}
  Let $S$ be a subalgebra of the algebra $C^\infty(\Xi)$. 
  If for each $(\+\xi_1, w_1), (\+\xi_2, w_2) \in T \Xi$, there 
  exists $f \in S$ such that 
  $(f(\+\xi_1), \inn{\partial f |_{\+\xi_1}, w_1}) \ne 
  (f(\+\xi_2), \inn{\partial f |_{\+\xi_2}, w_2})$, 
  then $S$ is dense in $C^\infty(\Xi)$. 
\end{lemma}

Under the assumption that $\+v \in C^\infty(\Xi, \bR^d)$ is a smooth 
injective immersion, one can obtain that the mapping 
\begin{displaymath}
  (\+\xi, w) \in T \Xi \mapsto (\+v(\+\xi), 
  \inn{\partial \+v|_{\+\xi}, w}) \in \bR^d \times \bR^d
\end{displaymath}
is injective. 
Let $S$ be the subalgebra of the algebra $C^\infty(\Xi)$ generated 
by $\{ v^j \}_{j = 1}^d$, i.e., 
\begin{displaymath}
  S = \mathrm{span} \left\{ v^{j_1} v^{j_2} \cdots v^{j_k} 
  \,\big|\, k \ge 0;\ j_1, j_2, \ldots, j_k = 1, 2, \ldots, d 
  \right\}.
\end{displaymath}
The subalgebra $S$ satisfies the assumption of 
\Cref{lemma:nachbin}, which yields that 
$S$ is dense in $C^\infty(\Xi)$.

\begin{proof}[Proof of \Cref{thm:bilinear}]
  Note that the direction \ref{item:bilinear_2} $\Rightarrow$ 
  \ref{item:bilinear_1} is straightforward. Let us prove 
  \ref{item:bilinear_1} $\Rightarrow$ \ref{item:bilinear_2}. 
  
  For each $h \in V \subset C^\infty(\Xi)$, 
  choose $\psi \in V$ such that $\psi \in [0, 1]$,
  and $\psi = 1$ on $\mathrm{supp}(h)$. Define 
  \begin{displaymath}
    \inn{\cA, h} = a(h, \psi).
  \end{displaymath}
  We assert that $\cA$ is well-defined. In other words, its 
  definition is 
  independent of the choice of $\psi$. Suppose that $\psi_1$ and 
  $\psi_2$ satisfy the conditions above. 
  Note that there exists a 
  sequence $\{ \phi_m \}_{m \in \bN}$ in $S$ such that $\phi_m$ 
  converges to $h$ in $C^\infty(\Xi)$ as $m$ tends to infinity.
  By \ref{item:bilinear_1}, the following equality holds,
  \begin{displaymath}
    a(\phi_m \psi_1, \psi_1 - \psi_2) = 
    a(\psi_1, \phi_m (\psi_1 - \psi_2)),
  \end{displaymath}
  and by letting $m$ tend to infinity, one obtains that 
  \begin{displaymath}
    a(h, \psi_1 - \psi_2) = 
    a(h \psi_1, \psi_1 - \psi_2) = 
    a(\psi_1, h (\psi_1 - \psi_2)) = 0.
  \end{displaymath}
  Assert that $\cA \in V^*$. Suppose that $\{ h_m \}_{m \in \bN}$
  is a sequence in $V$ such that $h_m$ converges to $h \in V$ as $m$ 
  tends to infinity. By definition, there exists a compact set 
  $K \subset \Xi$ such that $h_m$ and $h$ vanish out of $K$.
  Choose $\psi \in V$ such that $\psi \in [0, 1]$, and $\psi = 1$ on 
  $K$. Therefore, 
  \begin{displaymath}
    \inn{\cA, h_m} = a(h_m, \psi) \to 
    a(h, \psi),
  \end{displaymath}
  in $V$ as $m$ tends to infinity. Let us prove 
  \ref{item:bilinear_2}.
  For each $h_1, h_2 \in V$, choose $\psi \in V$ such that 
  $\psi \in [0, 1]$, and $\psi = 1$ on 
  $\mathrm{supp}(h_1) \cup \mathrm{supp}(h_2)$.
  There exists a sequence $\{ \phi_m \}_{m \in \bN}$ in $S$ such that 
  $\phi_m$ converges to $h_1$ in $C^\infty(\Xi)$, and thus, 
  $\phi_m \psi$ and $\phi_m h_2$ converge to $h_1$ and $h_1 h_2$ in 
  $V$, respectively, as $m$ tends to infinity. Note that 
  \begin{displaymath}
    a(\phi_m \psi, h_2) = a(\psi, \phi_m h_2),
  \end{displaymath}
  which yields that 
  \begin{displaymath}
    a(h_1, h_2) = a(\psi, h_1 h_2) = 
    \inn{\cA, h_1 h_2},
  \end{displaymath}
  by letting $m$ tend to infinity.
\end{proof}

  \bibliographystyle{acm}
  \bibliography{ref}


\end{document}